\documentclass[reqno,11pt]{amsart}
\usepackage[all]{xy}

\usepackage{amsthm,amsfonts, amssymb, amscd}
\usepackage{euscript}

\newtheorem{Thm}[equation]{Theorem}
\newtheorem{Lem}[equation]{Lemma}
\newtheorem{Cor}[equation]{Corollary}
\newtheorem{Prop}[equation]{Proposition}

\theoremstyle{remark}
\newtheorem{Rem}[equation]{Remark}
\theoremstyle{remark}

\theoremstyle{definition}
\newtheorem{Def}[equation]{Definition}

\numberwithin{equation}{section}

\newcommand{\C}{\mathbb{C}}           
\newcommand{\F}{\mathbb{ F}}           

\newcommand{\Hom}{\operatorname{Hom}}

\renewcommand{\ker}{\operatorname{ker }}

\newcommand{\im}{\operatorname{im }}

\newcommand{\fb}{{\mathfrak b}}

\newcommand{\fg}{{\mathfrak g}}
\newcommand{\fh}{{\mathfrak h}}
\newcommand{\fk}{{\mathfrak k}}
\newcommand{\fl}{{\mathfrak l}}

\newcommand{\fp}{{\mathfrak p}}

\newcommand{\fu}{{\mathfrak u}}

\newcommand{\fz}{\mathfrak z}

\newcommand{\f}{\mathfrak}

\newcommand{\ga}{\alpha}
\newcommand{\gb}{\beta}
\newcommand{\gd}{\delta}
\newcommand{\gD}{\Delta}

\newcommand{\gre}{\epsilon}

\newcommand{\gL}{\Lambda}

\newcommand{\gs}{\sigma}

\newcommand{\gt}{\theta}

\renewcommand{\tilde}{\widetilde}
\newcommand{\tilz}{\tilde{z}}
\newcommand{\spover}{{\overline{s}}_p}
\newcommand{\stotal}{\overline{s}}

\newcommand{\ua}{\underline{a}}

\newcommand{\ux}{\underline{x}}
\newcommand{\uy}{\underline{y}}
\renewcommand{\bar}[1]{\overline{#1}}

\begin{document}
\parskip=4pt
\baselineskip=14pt

\title[Relative Hochschild-Serre]{The relative Hochschild-Serre spectral sequence and the Belkale-Kumar product}

\author{Sam Evens}
\address{
Department of Mathematics,
University of Notre Dame,
Notre Dame, IN 46556
}
\email{sevens@nd.edu}

\author{William Graham}
\address{
Department of Mathematics,
University of Georgia,
Boyd Graduate Studies Research Center,
Athens, GA 30602
}
\email{wag@math.uga.edu}
\thanks{\noindent Evens was supported by the National Security Agency.}

\begin{abstract}
We consider the Belkale-Kumar cup product $\odot_t$ on $H^*(G/P)$ for
a generalized flag variety $G/P$ with parameter $t \in \C^m$,
where $m=\dim(H^2(G/P))$.   For each $t\in \C^m$, we define
an associated parabolic subgroup $P_K \supset P$.   We show that the ring $(H^*(G/P), \odot_t)$
contains a graded subalgebra $A$ isomorphic to $H^*(P_K/P)$ with the
usual cup product, where $P_K$ is a parabolic subgroup associated
to the parameter $t$.   Further, we prove that $(H^*(G/P_K), \odot_0)$
is the quotient of the ring $(H^*(G/P), \odot_t)$ with respect to
the ideal generated by elements of positive degree of $A$.
 We prove the above results by using basic facts about
the Hochschild-Serre spectral
sequence for relative Lie algebra cohomology, and most of the
paper consists of proving  these facts using the original
approach of Hochschild and Serre.  
\end{abstract}

\maketitle

\noindent 

\section{Introduction} \label{s.intro}
Let $G$ be a complex semisimple algebraic group with parabolic
subgroup  $P$ of $G$.  
Let $m=\dim(H^2(G/P))$.  For each $t\in \C^m$,
 Belkale and Kumar defined a product $\odot_t$  which degenerates the
usual cup product on $H^*(G/P)$, 
 and gave striking applications of this product to the eigenvalue problem and
to the problem of finding $G$-invariants in tensor products of representations
\cite{BeKu:06}.
In \cite{EG:11}, we gave a new construction of this product, and showed
that the ring $(H^*(G/P,\C), \odot_t)$ is isomorphic to a relative Lie algebra
cohomology ring $H^*(\fg_t, \fl_{\gD})$; here $\fg_t \supset \fl_{\gD}$ are certain
subalgebras of $\fg \times \fg$.  

In this paper, we study the ring $H^*(G/P, \odot_t)$.   Let
$\ga_1, \dots, \ga_n$ be the simple roots with respect to a 
Borel subgroup $B \subset P$ and a Cartan subgroup $H \subset B$. 
 Let $L$ be the Levi factor of $P$ containing $H$, let $\fl$ be
the Lie algebra of $L$, and number the simple roots so that
$I=\{ \ga_{m+1}, \dots, \ga_n\}$ are roots of $\fl$, and $\ga_1, \dots,
\ga_m$ are roots of $\fu$, the nilradical of the Lie algebra of $P$.
For each $t = (t_1, \ldots, t_m) \in \C^m$, let 
$J(t)=\{ 1 \le q \le m : t_q \not= 0 \}$, and let $K=J(t) \cup I$.
Let $\fl_K$ be the Levi subalgebra generated by the Lie algebra
$\fh$ of $H$ and the root spaces $\fg_{\pm \ga_i }$ for $i\in K$,
and let $L_K$ be the corresponding subgroup.  Let $P_K=BL_K$ be
the corresponding standard parabolic.

\begin{Thm} \label{t.borelextension}
For parabolic subgroups $P \subset P_K$ of $G$, with $P_K$ determined
by $t\in \C^m$ as above, 
\par\noindent (1)  
$H^*(P_K/P)$  is isomorphic to a graded subalgebra $A$
of $(H^*(G/P), \odot_t)$.  
\par\noindent (2) The ring $(H^*(G/P_K), \odot_0) \cong
(H^*(G/P), \odot_t)/I_+$, where $I_+$ is the ideal of
$(H^*(G/P), \odot_t)$ generated by positive degree elements
of $A$.
\end{Thm}

This theorem  asserts in effect that $(H^*(G/P), \odot_t)$
has a classical part isomorphic to the usual cohomology ring,
with associated quotient given by the degenerate Belkale-Kumar
product.
Using the relative Lie algebra cohomology description of the
product, this theorem follows from the
Hochschild-Serre spectral sequence in relative Lie algebra cohomology.
To make this argument rigorous, one needs to
know certain facts about the spectral sequence.  In particular, one must
show that the spectral sequence degenerates at the $E_2$-term, compute
the edge morphisms, and determine the product structure on the $E_2$-term.
The main point of this paper
is to carry this out using the original approach of Hochschild
and Serre \cite{HS:53}. 

While Borel and Wallach give an excellent treatment of representation
theoretic aspects of the Hochschild-Serre
spectral sequence in \cite{BoWa:00}, they do not discuss the ring 
structure,
and we are aware of no general reference for the facts needed to prove
our theorem.   The Hochschild-Serre spectral sequence
is of general utility, and one of the goals of this paper is to provide a careful treatment
of these facts for the literature.  Also,
the construction of the relative Hochschild-Serre spectral sequence
in \cite{BoWa:00} uses the identification of the relative Lie algebra
cohomology groups with derived functors.  
In \cite{EG:11} the definition of these
groups using cochains was crucial in order to be able to apply results
of Kostant from \cite{Kos:63}.   By proving what we need using the
cochain definition, we avoid the necessity of defining the ring structure
in the Borel-Wallach setting, and of proving the compatibility of the
two settings.  

The approach of Hochschild and Serre generalizes in a mostly straightforward
fashion to the relative setting, but there is one new point.  In this setting
we have a Lie algebra $\fg$, an ideal $I$ and a subalgebra
$\fk$ which is reductive in $\fg$.  To construct the spectral sequence we need
an action of $\fg/I$ on the relative cohomology group
$H^*(I, I \cap \fk; M)$ (here $M$ is a $\fg$-module).  If $I \cap \fk$ is nonzero,
the Lie algebra $\fg/I$ does not act in an obvious way on the space
of cochains $C^*(I, I \cap \fk; M)$.  Nevertheless, we are able to define the action
on the cohomology group 
by a formula involving cochains; verifying that this does yield the $d_1$ differential in the
spectral sequence is the main technical complication of the paper.
In fact, in the Belkale-Kumar application, $I \cap \fk = 0$, so this complication
can be avoided, but it seemed worthwhile to develop the spectral sequence
without this potentially limiting assumption.

The contents of the paper are as follows.  In Section \ref{s.prelim} we give
basic facts about a Lie algebra $\fg$ equipped with $I$ and $\fk$.  We also
recall some basic definitions related to Lie algebra cohomology, and
define the action of $\fg/I$ on
$H^*(I, I \cap \fk; M)$.  In Section \ref{s.filtration} we study the filtration
on cochains introduced by Hochschild and Serre and extend their results to
the relative setting.   In Section \ref{s.differentials}, we prove some formulas
involving differentials which are used in calculating the differentials in the
spectral sequence.  Section \ref{s.spectral} recalls basic definitions
and facts about spectral sequences.  Section \ref{s.hochserre}
proves the existence of the Hochschild-Serre spectral sequence in the relative
setting, identifies the edge maps and  proves some results on the product
structure.  Finally, Section \ref{s.bk} applies these results to the Belkale-Kumar
product.

\section{Preliminaries} \label{s.prelim}
We work over a field $\F$ of arbitrary characteristic.  In Section \ref{s.bk}, we take
$\F = \C$.
Let $\fg$ be a Lie algebra, $I \subset \fg$ an ideal, and $\fk \subset \fg$ a subalgebra.
Write $I_{\fk} = I \cap \fk$.
We assume that $\fk$ is reductive in $\fg$.  This is equivalent to
the statement that if $U \subset V$ are $\fk$-submodules of $\fg$, then
$U$ has a $\fk$-invariant complement $U'$, so $V = U \oplus U'$.

\begin{Lem} \label{l.reductive}
There is a $\fk$-module decomposition of $\fg$ given by
$$
\fg = I_{\fk} \oplus I_L \oplus J_{\fk} \oplus J_L,
$$
such that
\begin{equation} \label{e.reductive}
I = I_{\fk} \oplus I_L \mbox{  and  }\fk = I_{\fk} \oplus J_{\fk}.
\end{equation}
Moreover, $J := J_{\fk} \oplus J_L$ commutes with $I_{\fk}$.
\end{Lem}

\begin{proof}
Let $I_L$ be a $\fk$-module complement of $I_{\fk}$ in $I$, and let
$J_{\fk}$ be a $\fk$-module complement of $I_{\fk}$ in $\fk$.
Then 
$$
I + \fk = I_{\fk} \oplus I_L \oplus J_{\fk}.
$$
Let $J_L$ be a $\fk$-module complement to $I+\fk$ in $\fg$.
The decomposition
$$
\fg = I_{\fk} \oplus I_L \oplus J_{\fk} \oplus J_L,
$$
satisfies \eqref{e.reductive}.   Moreover, $[J, I_{\fk}] \subset J \cap I = 0$, so
$J$ and $I_{\fk}$ commute.
\end{proof}

We do not
assert that $J$ is a subalgebra of $\fg$, only a $\fk$-submodule.
We have a $\fk$-module decomposition $\fg = I \oplus J$.
Let $\pi: \fg \to I$ denote the projection arising
from this decomposition.  We will sometimes write $x^*$ for $\pi(x)$.
Let $x \mapsto x^+ = x - x^*$ denote the projection $\fg \to J$.

\begin{Cor} \label{c.reductive}
The map $\pi: \fg \to I$ is $\fk$-equivariant, so
if $x \in \fk$, $y \in \fg$, then $[x, \pi(y) ] = \pi([x,y])$,
or in other words, $[x, y^*] = [x,y]^*$.  Moreover, 
$\pi |_I = \mbox{id}$ and $\pi(\fk) = I_{\fk}$.
\end{Cor}

The proof is immediate.

\subsection{Lie algebra cohomology}
In this section we recall some of the basic definitions
of Lie algebra cohomology.
Let $M$ be a $\fg$-module, and let
$$
C^n(\fg;M) = \Hom(\gL^n \fg,M).
$$
We identify this space with the space of alternating $n$-linear maps
from $\fg \times \cdots \times \fg$ to $M$.
Given $\underline{x} = (x_1, \ldots, x_n) \in {\fg}^n$ and $z \in \fg$,
write 
$$
f(\ux) = f(x_1, \ldots, x_n)
$$
and
$$
f(z \cdot \underline{x}) = \sum_{i=1}^n f(x_1, \cdots, [z, x_i], \cdots, x_n).
$$

The Lie algebra cohomology differential $d: C^n(\fg;M) \to C^{n+1}(\fg;M)$ is
\begin{eqnarray*} \label{e.Liealgdiff}
df(x_0, \ldots, x_n) & = & \sum_{i \le n} (-1)^i x_i  f(x_0, \ldots ,\hat{x}_i , \ldots, x_n) \\
& + & \sum_{i<j \le n} (-1)^{i+j} f([x_i, x_j], x_0, \ldots,  \hat{x}_i \ldots  \hat{x_j} \ldots, x_n).
\end{eqnarray*}

We remark that for any vector space $V$ and any subspace
$W \subset V$ we always view $C^n(\fg/I; V) = \Hom_{\F}(\gL^n (\fg/I), V)$ as a subspace
of $C^n(\fg;V)$, so as arguments to $f \in C^n(\fg/I;V)$, we can allow elements of
$\fg$ or of $\fg/I$.  We adopt the analogous viewpoint for other spaces.

Define $\gt_z: C^n(\fg;M) \to C^n(\fg;M)$ and $i_x: C^n(\fg;M) \to C^{n-1}(\fg;M)$ by
\begin{eqnarray}
(\gt_z f)(\ux) & = & z \cdot (f(\ux)) - f(z \cdot \ux )   \\
(i_z f)(x_2, \ldots, x_n) & =  & f(z, x_2, \ldots, x_n).
\end{eqnarray}
Then
\begin{equation} \label{e.gt}
\gt_z = d i_z + i_z d.
\end{equation}
We will often write $z \cdot f$ for $\gt_z f$.
Observe that since $I$ is an ideal of $\fg$,
the same formula defines $\gt_z: C^n(I;M) \to C^n(I;M)$.

\begin{Lem} \label{l.dcommute}
The differentials on $C(\fg;M)$ and on $C(I;M)$ commute with the $\fg$-action.
\end{Lem}

\begin{proof}
The pullback $\gamma^*: C(\fg;M) \to C(I;M)$ is surjective and commutes
with $d$ and with the $\fg$ action.  Therefore it suffices to check that
the $\fg$-action on $C(\fg;M)$ commutes with $d$.  The element
$z \in \fg$ acts by $\gt_z$, and the fact that $\gt_z$ commutes with $d$
is an immediate consequence of \eqref{e.gt}.
\end{proof}

Let $C^n(\fg, \fk;M)$ denote the subspace of $C^n(\fg;M)$ consisting of those elements
annihilated by $i_x$ and $\gt_x$ for all $x \in \fk$, i.e., the subset
of linear maps $f: \Lambda^n(\fg/\fk) \to M$ such that
$$
\gt_x \cdot f = 0
$$
for all $x \in \fk$.  In other words, we can identify
$$
C^n(\fg, \fk;M) = \Hom_{\fk}(\gL^n(\fg/\fk);M).
$$
Let
$$
C(\fg, \fk;M) = \oplus_n C^n(\fg, \fk;M).
$$
It is well-known that
$C(\fg, \fk;M)$ is a subcomplex of $C(\fg;M)$.

Let $Z^q(\fg, \fk;M)$ and $B^q(\fg, \fk;M)$ denote the groups of cocycles and
coboundaries in $C^q(\fg, \fk;M)$.  
The $q$-th cohomology group of the complex $C(\fg, \fk;M)$ is denoted $H^q(\fg, \fk;M)$.
In the absolute case, $Z^q(\fg;M)$ and $B(\fg;M)$ denote the groups of cocycles and
coboundaries in $C^q(\fg;M)$, and $H^q(\fg;M)$ the cohomology.  
We adopt analogous notation for other Lie algebras.  

Pairings on relative Lie algebra cohomology are discussed in Section \ref{s.pairings}.

\subsection{The $\fg/I$-action on $H^*(I,I_{\fk};M)$}
In this section we define an action of $\fg/I$ on $H^q(I,I_{\fk};M)$ using
the complex $C(I,I_{\fk};M)$ (Borel and Wallach define a $\fg/I$-action
using injective resolutions).

\begin{Lem} \label{l.Jpreserve}
If $x \in J$, then the operator $\gt_x$ on $C^q(I;M)$ preserves the subspace $C^q(I,I_{\fk};M)$.
\end{Lem}

\begin{proof}
Let $c \in C^q(I,I_{\fk};M)$.  We must check that for $y \in I_{\fk}$, we have
$\gt_y \gt_x c = 0$ and $i_y \gt_x c = 0$.  The first equation holds because
$J$ and $I_{\fk}$ commute (Lemma \ref{l.reductive}); thus, $\gt_y$ and $\gt_x$ commute,
so $\gt_y \gt_x c = \gt_x \gt_y c = 0$, as $\gt_y c = 0$.  To check the second
equation, let $y_1 \in I_{\fk}$, and $y_2, \ldots, y_q \in I$.  Then 
$$
(i_{y_1} \gt_x c)(y_2, \ldots, y_q ) = x( c(y_1,  \ldots, y_q ))  + 
\sum_{i=1}^q c(y_1, \ldots, [y_i, x], \ldots, y_q).
$$
The term $c([y_1,x], y_2, \ldots, y_q)$ is zero because $[y_1,x]=0$.  All other terms
have $y_1$ as an argument of $c$; as $i_{y_1} c = 0$, these terms give $0$.
\end{proof}

If $x \in J$, then $\gt_x$ commutes with the differential $d$
on $C^q(I;M)$ (by Lemma \ref{l.dcommute}) and preserves the subspace
$C^q(I, I_{\fk};M)$ of $C^q(I;M)$.  Hence $\gt_x$ preserves the subspaces
$Z^q(I, I_{\fk};M)$ and $B^q(I,I_{\fk};M)$, so there is an induced
action of $\gt_x$ on the cohomology
$H^q(I,I_{\fk};M)$. 

As noted previously, we will often write simply $x \cdot c$ for $\gt_x c$.
Recall the projection $\fg \to J$, $x \mapsto x^+$.

\begin{Prop} \label{p.liealgaction}
There is a Lie algebra action of $\fg$ on $H^q(I,I_{\fk};M)$, defined by the formula
$ x \cdot [c] := [x^+ \cdot c]$ (where $x^+ \cdot c$ means $\gt_{x^+} c$).
Since $I$ acts trivially, this formula defines an action of $\fg/I$ on $H^q(I,I_{\fk};M)$.
\end{Prop}

\begin{proof}
We must check that the above definition is compatible with the Lie bracket, i.e., that
for $x,y \in \fg$, $c \in Z^q(I,I_{\fk};M)$, we have
\begin{equation} \label{e.lieaction1}
x \cdot (y \cdot [c]) - y \cdot (x \cdot [c]) = [x,y] \cdot [c].
\end{equation}
The left hand side is $[(x^+ \cdot y^+ - y^+ x^+ ) c] = [ [x^+,y^+] \cdot c]$.  The
right hand side is $[[x,y]^+ \cdot c]$.  Let $u = [x^+,y^+] - [x,y]^+$.  We must show
that $[u \cdot c] = 0$, that is, that $u \cdot c \in B^q(I,I_{\fk};M)$.

First, observe that $u \in I$.  Indeed,
$$
[x,y] = [x^* + x^+, y^* + y^+] = [x^* ,y^*  ] + [x^* ,y^+  ] + [x^+ ,y^*  ] + [x^+ ,y^+  ].
$$
On the right hand side, all terms but the last are in $I$; hence $ [x,y] - [x^+ ,y^+  ] \in I$.
On the other hand, $[x,y] - [x,y]^+ \in I$.  Therefore, $u =  ([x,y] - [x,y]^+) - ([x,y] - [x^+ ,y^+  ] ) \in I$.

Next, observe that $u$ commutes with $I_{\fk}$, as by
Lemma \ref{l.reductive}, $x^+, y^+$, and $[x,y]^+$ all commute with
$I_{\fk}$.    We conclude that $u \in Z_I (I_{\fk})$.

Since $dc = 0$, we have
$$
u \cdot c = ( d i_u c - i_u d c) = d i_u c.
$$
Thus, $u \cdot c \in B^q(I; M)$, and to show that
 $u \cdot c \in B^q(I,I_{\fk}:M)$, we must show that the element
$i_u c$ of $C^{q-1}(I;M)$ lies in the subspace $C^{q-1}(I,I_{\fk};M)$.  In other words,
we must show that for $a \in I_{\fk}$, we have $i_a i_u c = 0$ and $\gt_a i_u c = 0$.  The first
equation holds because $i_a i_u c = - i_u i_a c = 0$ as $i_a c = 0$.   To verify the second
equation, observe that
$$
\gt_a i_u c = (d i_a + i_a d) i_u c = i_a d i_u c = i_a \gt_u c,
$$
where in the second equality, we have used $i_a i_u c = 0$, and
in the third equality, we have used $dc = 0$.   If $y_2, \ldots, y_q \in I$,
then
\begin{eqnarray*}
(i_a \gt_u c)(y_2, \ldots, y_q) = (\gt_u c)(a,y_2, \ldots, y_q) = & u( c(a,y_2, \ldots, y_q) ) + 
c([a,u], y_2, \ldots, y_q)  \\ & + \sum_{i=2}^q c(a, y_2, \ldots, [y_i,a], \ldots, y_q)).
\end{eqnarray*}
This is $0$ because $i_a c = 0$ (so we get $0$ if any input to $c$ is equal to $a$)
and $[a,u] = 0$.  Hence $\gt_a i_u c = 0$.  We conclude that
$i_u c \in C^{q-1}(I,I_{\fk};M)$, so $\gt_u c = d i_u c \in B^q(I, I_{\fk};M)$, completing
the proof.
\end{proof}

\begin{Prop} \label{p.pullbackinvariant}
Let $j: (I,I_{\fk}) \to (\fg, \fk)$ denote the inclusion.  The pullback
$j^*: H^n(\fg, \fk;M) \to H^n(I,I_{\fk};M)$ has image in $H^n(I,I_{\fk};M)^{\fg/I}$.
\end{Prop}

\begin{proof}
Let $c \in C^n(\fg, \fk;M)$ be a cocycle.  We must show that $j^*[c]$ is $\fg/I$-invariant,
i.e.~ if $x \in \fg$, then $x \cdot (j^*[c]) = 0.$  By definition,
$$
x \cdot (j^*[c]) = [\gt_{x^+} j^* c].
$$
So we want to show that $\gt_{x^+} j^* c$ is the coboundary of an element in $C^{n-1}(I,I_{\fk};M)$.
We have
\begin{equation} \label{e.pullbackinvariant}
\gt_{x^+} j^* c = j^* \gt_{x^+}  c,
\end{equation}
provided the equation is interpreted correctly: $c \in C^n(\fg, \fk;M) \subset C^n(\fg;M)$,
but since we do not know that $C^n(\fg, \fk;M)$ is stable under $\gt_{x^+}$, we must
view $ \gt_{x^+}  c$ as an element of $C^n(\fg;M)$.  We have
$$
j^*  \gt_{x^+}  c = j^* (d i_{x^+} + i_{x^+} d) c = j^* d i_{x^+} c = d j^* i_{x^+} c.
$$
We know that $j^* i_{x^+} c$ is an element of $C^{n-1}(I;M)$.  To complete the proof,
we must show that
it lies in the subspace $C^{n-1}(I,I_{\fk};M)$.  Let $y \in I_{\fk}$.  Then
$i_y  j^* i_{x^+} c = - j^* i_{x^+} i_y c$, and this is zero
since $c \in C^n(\fg,\fk ;M)$.  Also, given $x_2, \ldots, x_n \in I$, we have
\begin{equation} \label{e.pullbackinvariant2}
(\gt_y j^* i_{x^+} c)(x_2, \ldots, x_n) = y \cdot( c(x^+, x_2, \ldots, x_n)) + \sum_{i \ge 2}
c(x^+, x_2, \ldots, [x_i, y], \ldots, x_n).
\end{equation}
Lemma \ref{l.reductive} implies that $[x^+, y] = 0$, so the right side of \eqref{e.pullbackinvariant2}
is unchanged if we add the term $c([x^+,y], x_2,  \ldots, x_n)$.  But if we add this term,
the right hand side of \eqref{e.pullbackinvariant2} is by definition equal to
$ (\gt_y c)(x^+, x_2, \ldots, x_n)$.  But $\gt_y c = 0$ since $c \in C^n(\fg, \fk;M)$.
We conclude that $j^* i_{x^+} c$ is in $C^{n-1}(I;M)$, completing
the proof.
\end{proof}

\section{The filtration and consequences} \label{s.filtration}
For simplicity, except for Section \ref{s.spectral}, we will now write
$C^n = C^n(\fg, \fk;M)$.  This space has a filtration (introduced by
Hochschild and Serre in the absolute setting): by definition,
$F_p C^n(\fg;M)$ consists of the subspace of $C^n(\fg;M)$ consisting of those $f$ which are zero
when $n-p +1$ of the inputs are in $I$.  In other words, $f(x_1, \ldots, x_n)$
can be nonzero only if at most $n-p$ of the $x_1, \ldots, x_n$ are
in $I$.  Given a subspace $V$ of $C^n(\fg;M)$, we define $F_p V = V \cap F_p C^n(\fg;M)$.
In particular, this defines $F_p C^n$.  As observed by Hochschild and Serre,
the Lie algebra cohomology differential
$d$ takes $F_p C^n(\fg;M)$ to $F_p C^{n+1}(\fg;M)$.   Hence $d$ takes $F_p C^n$ to $F_p C^{n+1}$.

We define $E_0^{pq} = F_p C^{p+q}/ F_{p+1}C^{p+q}$.  
This is the $0$-th page of the Hochschild-Serre spectral sequence (we recall
spectral sequence generalities in Section \ref{s.spectral}).
Our immediate goal is to describe $E_0^{pq}$.

Since $\f\fk/I_{\fk}$
acts on $\fg/I$, we can define the vector space
\begin{equation}
C^p(\fg/I, \fk/I_{\fk}; C^q(I, I_{\fk};M)) = \Hom_{\fk/I_{\fk}}(\gL^p \Big( \frac{\fg/I}{\fk/I_{\fk}} \Big), C^q(I, I_{\fk};M)).
\end{equation}
Note that
$$
 \frac{\fg/I}{\fk/I_{\fk}} \cong \fg /(I + \fk),
 $$
so
$$
C^p(\fg/I, \fk/I_{\fk}; C^q(I, I_{\fk};M)) = \Hom_{\fk/I_{\fk}}(\gL^p ( \fg /(I+\fk) ), C^q(I, I_{\fk};M)).
$$

\begin{Def} \label{d.rj}
Let $R_p$ denote the map $C^{p+q}(\fg;M) \to C^p(\fg; C^q(I;M))$ defined by
$$
(R_p f)(x_1, \ldots, x_p)(x_{p+1}, \ldots, x_{p+q}) = f(x_1, \ldots,x_{p+q})
$$
where $x_1, \ldots, x_p \in \fg$ and $x_{p+1}, \ldots, x_{p+q} \in I$.  Let
$S_p$ denote the restriction of $R_p$ to the subspace $C^{p+q} = C^{p+q}(\fg, \fk;M)$
of $C^{p+q}(\fg;M)$.  Let $r_p$ denote the restriction of
$R_p$ to $ F_p  C^{p+q}(\fg;M)$, and $s_p$ the restriction of
$S_p$ to $F_p C^{p+q}$.
\end{Def}

\begin{Lem} \label{l.restrict}
$S_p$ takes $C^{p+q}$ to $C^p(\fg, \fk;C^q(I;M))$.
\end{Lem}

\begin{proof}
We must show that if $f$ is $\fk$-invariant, then so is
$S_p f$.
Let
$z \in \fk$, $x_1, \ldots, x_p \in \fg$, and $y_1, \ldots, y_q \in I$.
Now, $\fk$-equivariance means that if $z \in \fk$ and $\ux \in {\fg}^p$,
then
\begin{equation} \label{e.slf3}
z \cdot [(S_p f)(\ux) ] = (S_p f)(z \cdot \ux).
\end{equation}
To verify this equation, we must evaluate both sides
at $\uy \in {\fg}^q$.  The left side gives
\begin{equation} \label{e.slf4}
-  f(\ux, z \cdot \uy) + z \cdot f(\ux, \uy),
\end{equation}
while the right side gives
\begin{equation} \label{e.slf5}
f( z \cdot \ux, \uy).
\end{equation}
The expressions \eqref{e.slf4} and \eqref{e.slf5} are equal because
$f$ is $\fk$-equivariant.  Hence $S_p f$ is $\fk$-equivariant,
as desired.
\end{proof}


We clarify some notation regarding interior products.
Note that if $f \in C^p(\fg, C^q(\fg;M))$ and $y\in \fg$, then $i_y f \in C^{p-1}(\fg, C^q(\fg;M))$.
Given a $p$-tuple $\underline{x} = (x_1, \ldots, x_j)$,
let $i_{\underline{x}}:   C^p(\fg; C^q(\fg;M)) \to C^q(\fg;M)$ denote
the evaluation map.  In this case, if $f \in C^p(\fg; C^q(\fg;M))$
then $i_y i_{\ux} f \in C^{q-1}(\fg;M)$, and $\gt_y i_{\ux} f \in C^{q-1}(\fg;M)$.
We will use the same notation for maps restricted
to subspaces of $C^{p+q}(\fg;M)$ and $C^p(\fg; C^q(\fg;M))$.

The following lemma generalizes to the relative situation a result
from \cite{HS:53}; the arguments are adapted from there.
The lemma implies that
$s_p$ induces an isomorphism
$$
E_0^{pq} = F_p C^{p+q} / F_{p+1} C^{p+q} \cong  C^p(\fg/I, \fk/I_{\fk}; C^q(I, I_{\fk};M)).
$$
This is our desired description of $E_0^{pq}$.

\begin{Lem} \label{l.sj}
\begin{enumerate}
\item The kernel of $s_p$ is $F_{p+1} C^{p+q}$. 
\item  $s_p$ maps $F_p C^{p+q}$ 
to $C^p(\fg/I, \fk/I_{\fk}; C^q(I, I_{\fk};M))$.  
\item The image of $s_p$ is all of $C^p(\fg/I, \fk/I_{\fk}; C^q(I, I_{\fk};M))$. 
\end{enumerate}
\end{Lem}

\begin{proof}
(1) The statement about the kernel of $s_p$ is immediate from the definition of the
filtration.  

(2) Suppose $f \in F_p C^{p+q}$.  We want to show that
$s_p(f) \in  C^p(\fg/I, \fk/I_{\fk}; C^q(I, I_{\fk};M))$.  Since $s_p$ is the restriction
of $S_p$ and of $r_p$, we know that $s_p f $ lies in the intersection
of the images of $S_p$ and of $r_p$, i.e., in
$$
C^p(\fg, \fk;C^q(I;M)) \cap C^p(\fg/I; C^q(I;M)).
$$
To prove (2), it will be enough to show that for
$\ux \in {\fg}^p$, we have $i_{\ux} f \in C^q(I, I_{\fk};M)$, for then
the notions of $\fk$-equivariance
and $\f\fk/I_{\fk}$-equivariance will coincide.  

To show that  $i_{\ux} f \in C^q(I, I_{\fk};M)$,
we must verify that for $z \in I_{\fk}$,
we have $$i_z  i_{\ux} s_p f = 0$$
and
$$\gt_z  i_{\ux}  s_p f =0.
$$
The first equation holds because if an element of $\fk$ is an argument
of $f$, the result is $0$.  For the second equation, since $f \in C^{p+q}(\fg, \fk;M)$,
and $z \in I_{\fk} \subset \fk $, we have $\gt_z f = 0$.  By definition of $\gt_z$,
this means that for $\ux \in {\fg}^p$ and $\uy \in I^q$, we have
$$
z \cdot (f(\ux, \uy)) = f(z \cdot \ux, \uy) + f(\ux, z \cdot \uy).
$$
The first term on the right is $0$ because it is a sum of terms, each of which
has at least $q+1$ inputs from $I$.   Thus, we obtain
$$
[ \gt_z  i_{\ux}  s_p f ] (\uy)= z \cdot (f(\ux, \uy)) - f(\ux, z \cdot \uy) = 0.
$$
As $\uy$ is arbitrary, we see that $\gt_z i_{\ux}  s_p f = 0$,
as desired.

(3) We now prove that $s_p$ is surjective.  We must show that if
$g \in C^p(\fg/I, \fk/I_{\fk}; C^q(I, I_{\fk};M))$, then there exists
$f \in F_p C^{p+q}$ such that $s_p f = g$.

As in \cite{HS:53}, define
$f \in C^{p+q}(\fg;M)$ by
\begin{equation} \label{e.tildemap}
f(x_1, \ldots, x_{p+q}) = \sum_{\gs} (-1)^{\gs} g(x_{\gs(1)}, \ldots, x_{\gs(p)}) 
(x_{\gs(p+1)}^*, \ldots, x_{\gs(p+q)}^*).
\end{equation}
Here, the map $x \mapsto x^*$ is as in Corollary 
\ref{c.reductive}.  The sum is over all
permutations $\gs \in \Sigma_{p+q}$ such that
$\gs(1) < \gs(2) < \cdots < \gs(p)$ and
$\gs(p+1) < \gs(p+2) < \cdots < \gs(p+q)$.  (We have asserted that
$f \in C^{p+q}(\fg;M)$; this amounts to checking that $f$ is alternating---see
Remark \ref{r.alternating} below.)
We must verify that $f \in C^{p+q}(\fg, \fk;M)$, i.e., that if $x \in \fk $,
then $i_x f = 0$ and $\gt_x f = 0$.  To show the first equation, it
suffices to show that if $x_1 \in \fk $, then $f(x_1, \ldots, x_{p+q}) = 0$.
This holds because either $\gs(1) = 1$, in which case we
get $0$ as $i_{x_1} g =0$; or $\gs(p+1) = 1$, in which case
$x_1^* \in I_{\fk}$ by Corollary \ref{c.reductive}, so we get $0$
because $i_{\ux} g \in C^q(I, I_{\fk};M)$.

Now we show that $\gt_x f = 0$.   This says that for $z \in \fk $
and $x_1, \ldots, x_{p+q} \in \fg$, we have
\begin{equation} \label{e.fequiv}
\sum_{i = 1}^{p+q} f(x_1 \ldots, [z, x_i], \ldots, x_{p+q} ) = z \cdot (f(x_1, \ldots, x_{p+q}) ).
\end{equation}
By assumption, $g$ is $\f\fk/I_{\fk}$-equivariant
(i.e., $\fk$-invariant, and $I_{\fk}$ acts trivially).   
Thus, for $z \in \fk $ and $x_1, \ldots, x_p \in \fg$, we have
\begin{equation} \label{e.gequiv1}
\sum_{i=1}^p g(x_1, \ldots, [z, x_i], \ldots, x_p) = z \cdot (g(x_1, \ldots, x_p ) ).
\end{equation}
Given $(x_{p+1}, \ldots, x_{p+q}) \in {\fg}^q$, we have
$(x_{p+1}^*, \ldots, x_{p+q}^*) \in I^q$, to which we can apply
both sides of the preceding equation.  Upon rearranging,
by definition of the $\fk$-action on $C^q(I, I_{\fk};M)$, we obtain
\begin{eqnarray*} 
\sum_{i=1}^p g(x_1,\ldots, [z, x_i], \ldots, x_p) (x_{p+1}^*, \ldots, x_{p+q}^*) 
+ \sum_{j=p+1}^{p+q} g(x_1, \ldots, x_p)(x_{p+1}^*\ldots, [z, x_j^*], 
\ldots , x_{p+q}^*) \\
 = z \cdot \big( g(x_1, \ldots, x_p) (x_{p+1}^*, \ldots, x_{p+q}^*) \big) .
\end{eqnarray*}
Now, $[z, x_j^*] = [z, x_j]^*$ by Corollary \ref{c.reductive}.
Hence
\begin{eqnarray*}
\sum_{\gs} (-1)^{\gs} \big( \sum_{i=1}^p g(x_{\gs(1)},  \ldots, [z, x_{\gs(i)}], \ldots,  x_{\gs(p)}) 
(x_{\gs(p+1)}^*, \ldots, x_{\gs(p+q)}^*) \\
+ \sum_{j=p+1}^{p+q} g(, \ldots, x_{\gs(p)})
(x_{\gs(p+1)}^*\ldots, [z, x_{\gs(j)}]^*, 
\ldots , x_{\gs(p+q)}^*)  \big) \\
= \sum_{\gs} (-1)^{\gs} z \cdot \big( g(x_{\gs(1)}, \ldots, x_{\gs(p)}) 
(x_{\gs(p+1)}^*, \ldots, x_{\gs(p+q)}^*) \big).
\end{eqnarray*}
This is exactly the equation we obtain by taking the equation
\eqref{e.fequiv} of $\fk$-equivariance of $f$ and expressing $f$
in terms of $g$.  This proves \eqref{e.fequiv}, so $f$ is
$\fk$-equivariant.  We conclude that $f \in C^{p+q}(\fg, \fk;M)$, as desired.

Next, observe that $f \in F_p C^{p+q}(\fg, \fk;M)$.  The reason is that
if more than $q$ of the arguments of $f$ are in $I$, then
for any permutation $\gs$,
at least one of the $x_{\gs(1)}, \ldots, x_{\gs(p)}$ must be in $I$,
in which case 
$g(x_{\gs(1)}, \ldots, x_{\gs(p)}) (x_{\gs(p+1)}^*, \ldots, x_{\gs(p+q)}^*) = 0$.

Finally, we must verify that $s_p f = g$.  For this, suppose
that $x_1, \ldots, x_p$ are in $\fg$ and $x_{p+1}, \ldots, x_{p+q}$ are
in $I$.  Then in the definition of $f$, the only permutation that
contributes is the identity permutation, and we see that
$$
f(x_1, \ldots, x_{p+q}) = g(x_1, \ldots, x_p)(x_{p+1}, \ldots, x_{p+q}),
$$
so $s_p f = g$, as desired.
\end{proof}

\begin{Rem} \label{r.alternating}
To show that $f$ is alternating, it suffices to show that if
$x_i = x_{i+1}$ for some $i$, then $f(x_1, \ldots, x_{p+q}) = 0$.
Suppose $\sigma \in \Sigma_{p+q}$ is a permutation
contributing to the sum \eqref{e.tildemap}, and suppose that
$\sigma(a) = i$, $\sigma(b) = i+1$.  If $a$ and $b$ are both $\le p$
or both $\ge p+1$, then $g(x_{\gs(1)}, \ldots, x_{\gs(p)}) 
(x_{\gs(p+1)}^*, \ldots, x_{\gs(p+q)}^*) = 0$ by the alternating properties of $g$.
Otherwise, the contribution 
to \eqref{e.tildemap} from $\sigma$ is the negative of the contribution from
 $\tau \sigma$, where $\tau$
is the transposition $(i \ i+1)$.  We conclude that $f(x_1, \ldots, x_{p+q}) = 0$, as desired.
\end{Rem}


\begin{Def} \label{d.tildemap}
For later use, we will denote by $g \mapsto \tilde{g}$ the map
$$
C^p(\fg/I, \fk/I_{\fk}; C^q(I, I_{\fk};M)) \to F_p C^{p+q}
$$
constructed
in the preceding proof; that is, $\tilde{g}$ is the element $f$
defined in \eqref{e.tildemap}.  The preceding
proof shows that $s_p \tilde{g} =  S_p(\tilde{g}) = g$.
\end{Def}

\section{Some formulas involving differentials} \label{s.differentials}
We give here some facts involving various Lie algebra
cohomology differentials, which will be used in Section \ref{s.hochserre}
when we calculate
the differentials in the spectral sequence.

We let $C^{*,*}=\oplus_{p, q \ge 0} C^{p,q}$, where $C^{p,q}=C^p(\fg;C^q(I;M)$.
Recall from Definition \ref{d.rj} the map
$R_p: C^{p+q}(\fg;M) \to C^{p,q}$.
There are two differentials
on $C^{*,*}$: first, there is
the vertical differential 
$d_v: C^{p,q}\to C^{p,q+1}$.
This is defined as $d_v f = d \circ f$ for $f\in C^p(\fg; C^q(I;M))$,
 where the $d$ on the right hand side of
the equation is the differential $C^q(I;M) \to C^{q+1}(I;M)$.  
Second, there is the horizontal differential
$d_h: C^{p,q} \to C^{p+1,q}$
(this is the Lie algebra
cohomology differential $d:C^p(\fg; C^q(I;M)) \to C^{p+1}(\fg; C^q(I;M))$,
 with $C^q(I;M)$ playing the role of the module;
note that $x \in \fg$ acts on $C^q(I;M)$ using the action $\gt_x$
defined above).

\begin{Lem} \label{l.dhdvcommute}
The differentials $d_h$ and $d_v$ (on $C(\fg, C(I;M))$) commute.
\end{Lem}

\begin{proof}
Any $\fg$-module map
$V \to W$ induces a map of complexes $C(\fg;V) \to C(\fg;W)$ which is compatible with
the Lie algebra cohomology differential.
By Lemma \ref{l.dcommute},
the differential $C^q(I;M) \to C^{q+1}(I;M)$ is a $\fg$-module map.  The
induced map $C(\fg;C^q(I;M)) \to C(\fg;C^{q+1}(I;M))$ is what we have denoted $d_v$;
the compatibility with the Lie algebra cohomology differential (with $\fg$ playing the role
of Lie algebra) amounts to the commutativity of $d_h$ and $d_v$.  
\end{proof}

\begin{Lem} \label{l.dstable}
The space $C(\fg, \fk;C(I,M))$ (viewed as a subspace
of $C(\fg,C(I,M))$) is stable
under $d_h$ and $d_v$.  Moreover, the spaces
$$
C(\fg/I, \fk/I_{\fk}; C(I, I_{\fk}; M)) \subset C(\fg, \fk;C(I,I_{\fk};M))
$$
of $C(\fg, \fk;C(I,M))$
are stable under $d_v$.
\end{Lem}

\begin{proof}
The space $C(\fg, \fk;C(I,M))$ is stable under $d_h$ because it is a relative Lie
algebra cohomology complex (with
$C(I;M)$ playing the role of module).   To verify stability under 
$d_v$, let $f \in C^p(\fg, \fk;C^q(I;M))$, and $z \in \fk $.  We must show two things.
First, for $z$ in $\fk$,
$d_v f \in C^p(\fg; C^{i+1}(I;M))$ is annihilated by $i_z$: this follows because
$d_v f = d \circ f$, and $f$ vanishes when any input is in $\fk$.  Second,
$d_v f: {\fg}^p \to C^q(I;M)$ is $\fk$-equivariant.  Indeed, suppose $z \in \fk $ and
$\ux \in {\fg}^p$.  Writing the action of $z$ on $C(I;M)$ as $\gt_z$,
we want to show that
$$
(d_v f)(z \cdot \ux) = \gt_z ( (d_v f)(\ux)).
$$
By definition, $(d_v f)(z \cdot \ux) = d (f( z \cdot \ux))$.
Since $f$ is $\fk$-equivariant, this equals
$d(\gt_z(f(\ux)))$.  By Lemma \ref{l.dcommute},
$d$ commutes with $\gt_z$.  Hence $C(\fg, \fk;C(I;M))$
is stable under $d_v$.  The remaining assertions follow
easily.
\end{proof}


The following lemma is essentially given in \cite{HS:53}, so we omit
the proof.

\begin{Lem} \label{l.differential}
Let $f \in C^{p+q}(\fg;M)$.  Then
\begin{equation} \label{e.differential1}
R_{p+1} df = d_h (R_p f) + (-1)^{p+1} d_v (R_{p+1}f).
\end{equation}
The analogous formula holds with $R$ replaced by $S$.
\end{Lem}

\subsection{Differentials related to relative cohomology}
To streamline the exposition we introduce some notation.

\begin{Def} \label{d.streamline}
Write
$$
C^p(C^q) = C^p(\fg/I, \fk/I_{\fk}; C^q(I, I_{\fk};M)).
$$
We adopt the analogous notation when $C^q(I, I_{\fk};M))$
is replaced by $Z^q(I, I_{\fk};M)$, $B^q(I, I_{\fk};M)$ or $H^q(I, I_{\fk};M)$
(cocycles, coboundaries, and cohomology, respectively).
When we write $C^p$ by itself, we will mean
$C^p(\fg, \fk;M)$.  Also, since $\fg/I$ acts on $H^q(I,I_{\fk};M)$,
we have groups of coboundaries and cocycles:
$$
B^p(H^q) := B^p(\fg/I,\fk /I_{\fk }; H^q(I,I_{\fk};M)) \subset 
Z^p(H^q) := Z^p(\fg/I,\fk /I_{\fk }; H^q(I,I_{\fk};M))
$$
and the corresponding cohomology group
$$
H^p(H^q) := H^p(\fg/I,\fk /I_{\fk }; H^q(I,I_{\fk};M)).
$$
\end{Def}

We need a little more notation.
Given a $(p+1)$-tuple $\uy = (y_0, \ldots, y_p)$, and $0 \le i \le p$,
let $\uy(i)$ denote the $p$-tuple $(y_0, \ldots,\hat{y}_i, \ldots y_p)$.
Given $z \in C^p(Z^q)$, we can view $z(\uy(j))$ as an element
of $Z^q(I, I_{\fk};M) \subset Z^q(I;M)$.  The space $Z^q(I;M)$ has an $\fg$-module
structure.  
To be explicit, given $u \in \fg$ and
$y_{p+1}, \ldots, y_{p+q} \in I$, we have
\begin{eqnarray*}
[u \cdot z(\uy(j))](y_{p+1}, \ldots, y_{p+q})  & = & u \cdot \big[ z(\uy(j))(y_{p+1}, \ldots, 
y_{p+q}) \big] \\
& + & \sum_{r = p+1}^{p+q} z(\uy(j))(y_{p+1}, \ldots,[y_r, u], \ldots, \ldots, y_{p+q}).
\end{eqnarray*}

\begin{Lem} \label{l.dplus}
There is a map 
$$
d_+: C^p(C^q) \to C^{p+1}(C^q)
$$
defined by the formula
$$
(d_+ z)(y_0, \ldots, y_p) = \sum_{i=0}^p (-1)^i y_i^+ \cdot z(\uy(i)) + 
\sum_{0 \le i<j \le p} (-1)^{i+j} z([y_i, y_j], y_0, \ldots, 
\hat{y}_i, \ldots, \hat{y}_j, \ldots, y_p). 
$$
\end{Lem}

\begin{proof}
It is evident that $d_+ z \in C^{p+1}(\fg/I; C^q(I,I_{\fk};M))$.  We must verify that
$d_+ z$ is $\f\fk/I_{\fk}$-equivariant, or equivalently, is $\fk$-equivariant.
We have
$$
C^p(\fg/I, \fk/I_{\fk}; C^q(I,I_{\fk};M)) \subset C^p(\fg, \fk;C^q(I;M)) \stackrel{d_h}{\rightarrow} C^{p+1}(\fg, \fk;C^q(I;M)).
$$
Define $e: C^p(\fg/I, \fk/I_{\fk}; C^q(I,I_{\fk};M)) \to  C^{p+1}(\fg/I; C^q(I;M))$ by the formula
\begin{equation} \label{e.defe}
(ez)(y_0, \ldots, y_p) = \sum_{i=0}^p (-1)^i y_i^* \cdot \big( z(\uy(i) \big)).
\end{equation}
Since $d_h$ is a relative Lie algebra cohomology differential, $d_h z$ is $\fk$-equivariant.
Since $d_+ = d_h - e$, it suffices to check that
$ez$ is $\fk$-equivariant as well, i.e., that if $u \in \fk $, and $\uy = (y_0, \ldots, y_p)$, then
\begin{equation} \label{e.dplus1}
u ( (ez)(\uy))= (ez)(u \cdot \uy).
\end{equation}
The left hand side of this equation is
\begin{equation} \label{e.dplus2}
\sum_{i=0}^p (-1)^i  u y_i^* z(\uy(i)).
\end{equation}
The right hand side is
\begin{equation} \label{e.dplus3}
\sum_{i=0}^p (ez)( y_0, \ldots, [u, y_i], \ldots, y_p).
\end{equation}
This can be rewritten as
\begin{equation} \label{e.dplus4}
\sum_{i = 0}^p (-1)^i [u,y_i]^* \cdot z(\uy(i)) + \sum_{i=0}^p \sum_{j \neq i}
(-1)^j y_j^* z(y_0, \ldots, \hat{y}_j, \ldots, [u, y_i], \ldots, y_p ).
\end{equation}
By Corollary \ref{c.reductive} $[u,y_i]^* = [u, y_i^*]$, and this element
acts on $z(\uy(i))$ as $uy_i^* - y_i^*u$.  
Also, the double sum in \eqref{e.dplus4} 
can be rewritten as
$$
\sum_{j = 0}^p (-1)^j y_j^* z(u \cdot \uy(j)).
$$
Therefore, changing $j$ to $i$ in this sum, we can rewrite \eqref{e.dplus4} as
$$
\sum_{i = 0}^p (-1)^i (uy_i^* - y_i^* u) \cdot z(\uy(i))  + \sum_{i = 0}^p (-1)^i y_i^* z(u \cdot \uy(i)).
$$
We want to show that this is equal to \eqref{e.dplus2}, i.e., that
$$
0 = - \sum_{i = 0}^p (-1)^i y_i^* u \cdot z(\uy(i))  + \sum_{i = 0}^p (-1)^i y_i^* z(u \cdot \uy(i)).
$$
This equality holds because the $\fk$-equivariance of $z$ implies that $u \cdot z(\uy(i)) = z(u \cdot \uy(i))$.
\end{proof}

Although $d_h$ commutes with $d_v$, we do not
assert that $d_+$ commutes with $d_v$.

Notice that the formula for $d_+$ would be a Lie algebra cohomology differential
if $\fg/I$ acted on $C^q(I,I_{\fk};M)$, with the action of $x$ taking $c$ to
$x^+ \cdot c$.  However, this formula does not define a Lie algebra action
(i.e., it is not compatible with the Lie bracket).
Thus, we cannot assert that $d_+ \circ d_+$ is zero on cochains.
However, by Proposition \ref{p.liealgaction}, $\fg/I$ does act on $H^q(I,I_{\fk};M)$
by the analogous formula, i.e., by $x \cdot [c] = [x^+ \cdot c]$.  
Write $[z]$ for the image of $z$ under the map
$$
C^p(\fg/I, \fk/I_{\fk}; Z^q(I,I_{\fk};M)) \to C^p(\fg/I, \fk/I_{\fk}; H^q(I,I_{\fk};M)).
$$
The Lie algebra cohomology differential
$$
C^p(\fg/I, \fk/I_{\fk}; H^q(I,I_{\fk};M)) \to C^{p+1}(\fg/I, \fk/I_{\fk}; H^q(I,I_{\fk};M))
$$
is given by $[z] \mapsto [d_+ z]$.  By abuse of notation we will simply
write $d_+ [z] = [d_+ z]$. 

\begin{Lem} \label{l.drearrange}
Let $f \in Z^q(I;M)$, and let $x_0, \ldots, x_q \in I$.  Then
\begin{eqnarray*}
(x_0 \cdot f)(x_1, \ldots, x_q) & = & \sum_{r=1}^q (-1)^{r+1} x_r (f(x_0, \ldots, \hat{x}_r, \ldots, x_q)) \\
& + & \sum_{1 \le r < s} (-1)^{r+s} f(x_0, [x_r, x_s],x_1,  \ldots, \hat{x}_r, \ldots, \hat{x}_s, \ldots).
\end{eqnarray*}
\end{Lem}

\begin{proof}
This follows by writing out the equation $df(x_0, \ldots, x_q) = 0$ and rearranging, using
the definition of $x_0 \cdot f$.
\end{proof}

The following proposition and corollary relate $z = S_p \tilz $ and $S_{p+1} \tilz$.

\begin{Prop} \label{p.dhtilde}
Let $z \in C^p(\fg/I, \fk/I_{\fk}; Z^q(I, I_{\fk};M))$, and let $\tilz \in C^{p+q}$ be defined
as in Definition \ref{d.tildemap} (so $S_p \tilz = z$).  
Let $(y_0, \ldots, y_p) \in {\fg}^{p+1}$.  Then
$$
(d_v S_{p+1} \tilz )(y_0, \ldots, y_p) = (-1)^p \sum_{i=0}^p  (-1)^i y_i^* \cdot z(\uy(i)).
$$
\end{Prop}

\begin{proof}
First, observe that given $\ua = (a_0, \ldots, a_p ) \in {\fg}^p$, and $(b_1, \ldots, b_{q-1}) \in I^{q-1}$, we have
\begin{equation} \label{e.dhtilde1}
( S_{p+1} \tilz)(a_0, \ldots, a_p )(b_1, \ldots, b_{q-1}) = 
(-1)^p \sum_{i=0}^p (-1)^i z(\ua(i))(a_i^*, b_1, \ldots,b_{q-1}) .
\end{equation}
This follows from the definition of $\tilz$, along with the fact that if one of
$u_1, \ldots, u_p$ is in $I$, then
$z(u_1, \ldots, u_p)(v_1, \ldots, v_q) = 0$.

Now let $(y_0, \ldots, y_p) \in {\fg}^{p+1}$ and $(y_{p+1}, \ldots, y_{p+q}) \in I^q$.
To simplify the notation, for each $i \le p$, let $F_i = z(\uy(i)) \in Z^q(I)$.  
By definition, 
$$
(d_v S_{p+1} \tilz )(y_0, \ldots, y_p)(y_{p+1}, \ldots, y_{p+q})
$$ 
is equal to
\begin{eqnarray*}
 \sum_{p+1 \le r} (-1)^{r+p+1}   y_r (S_{p+1} \tilz )(y_0, \ldots, y_p)(y_{p+1}, \ldots,\hat{y}_r, \ldots, y_{p+q})  \\
+  \sum_{p+1 \le r<s} (-1)^{r+s}
(S_{p+1} \tilz )(y_0, \ldots, y_p)([y_r, y_s], y_{p+1}, \ldots,\hat{y}_r, \ldots, \hat{y}_s, 
\ldots, y_{p+q}).
\end{eqnarray*}
Equation \eqref{e.dhtilde1} implies that this is equal to
\begin{eqnarray*}
\sum_{p+1 \le r} \sum_{i \le p} (-1)^{r+i+1} y_r F_i (y_i^*, y_{p+1}, \ldots, \hat{y}_r, \ldots)  \ \ \ \ \ \ \ \ \ \ \ \ \ \ \ \ \ \ \ \   \\
+ \sum_{p+1 \le r<s} \sum_{i \le p} (-1)^{r+s+i+p} F_i (y_i^*, [y_r, y_s], y_{p+1},
 \ldots,\hat{y}_r, \ldots, \hat{y}_s, \ldots, y_{p+q} ). 
\end{eqnarray*}
We can rearrange this expression to obtain
\begin{eqnarray*}
(-1)^p \sum_{i \le p} (-1)^i \Big( \sum_{p+1 \le r} (-1)^{r+1+p} y_r ( F_i (y_i^*, y_{p+1}, \ldots, \hat{y}_r, \ldots, y_{p+q})  ) \ \ \ \ \ \ \ \ \ \ \ \ \ \ \ \ \ \ \ \   \\
+  \sum_{p+1 \le r<s} (-1)^{r+s} F_i (y_i^*, [y_r, y_s], \ldots,\hat{y}_r, \ldots, \hat{y}_s, \ldots,y_{p+q}) \Big). 
\end{eqnarray*}
It follows from Lemma \ref{l.drearrange}, with $F_i$ playing the role of $f$,
and $y_i^*, y_{p+1},  \ldots, y_{p+q}$ playing the role of
$x_0, \ldots, x_q$, 
that the above expression is equal to
$$
(-1)^p \sum_{i=0}^p (-1)^i (y_i^* F_i)(y_{p+1}, \ldots,y_{p+q})
$$
proving the proposition.
\end{proof}

\begin{Cor} \label{c.sp+1}
Let $z$ and $\tilz$ be as in Proposition \ref{p.dhtilde}.  Then
$$
s_{p+1} d \tilz = d_+ z.
$$
\end{Cor}

\begin{proof}
By definition,  $s_{p+1} d \tilz = S_{p+1} d\tilz$.
By construction, $S_p \tilz = z$.  Thus, by Lemma \ref{l.differential}, 
$$
S_{p+1} d \tilz= d_h z + (-1)^{p+1} d_v (S_{p+1} \tilz).
$$
Let $(y_0, \ldots, y_p) \in {\fg}^{p+1}$.  By definition of $d_h$, we have
\begin{equation} \label{e.bars1}
d_h z(y_0, \ldots, y_p) = \sum_{i=0}^p (-1)^i y_i \cdot z(\uy(i)) + 
\sum_{0 \le r<s \le p} (-1)^{r+s} z([y_r, y_s], y_1, \ldots, \hat{y}_r, \ldots, \hat{y_s}, \ldots).
\end{equation}
On the other hand, by Proposition \ref{p.dhtilde},
\begin{equation} \label{e.bars2}
 (-1)^{p+1} d_v (S_{p+1} \tilz)(y_0, \ldots, y_p) = - \sum_{i=0}^p  (-1)^i y_i^* \cdot z(\uy(i)).
\end{equation}
Now, $y_i - y_i^* = y_i^+$.  The corollary follows by
adding \eqref{e.bars1} and \eqref{e.bars2} and comparing the result
with the definition of $d_+ z$.
\end{proof}

\section{The spectral sequence} \label{s.spectral}

\subsection{Spectral sequence generalities}
In this section we recall some standard facts about spectral sequences.
Our basic source is Chapter 5 of \cite{Wei:94}, but we have modified some of the
definitions for convenience.  See also Chapter XX.9 of \cite{Lan:02}.

Suppose $C = \oplus_{n \ge 0} C^n$ is a graded cochain complex with
differential $d$.  We assume also that $C$ has a decreasing filtration (compatible with $d$):
that is, for each $n$, we have
$$
F_0 C^n \supset F_1 C^n \supset \cdots ;
$$
the filtration is extended to negative indices by setting $F_p C^n = F_0 C^n$ for $p<0$.
We also assume
that for each $n$ there exists some $r$ (depending on $n$) such that $F_r C^n = 0$.
Let $B^n \subset Z^n$ denote the spaces of coboundaries
and cocycles (respectively) in $C^n$.  These spaces are filtered by setting
$F_p B^n := F_p C^n \cap B^n$ and $F_p Z^n := F_p C^n \cap Z^n$.
The cohomology $H^n(C)$ is filtered by $F_p H^n (C) = (F_p Z^n + B^n) / B^n$.

Define
$$
E_0^{pq} = \mbox{gr}^p C^{p+q} :=F_p C^{p+q}/F_{p+1} C^{p+q},
$$
and let $\pi_p: F_p C^{p+q} \to E_0^{pq}$ denote the projection. 

For each $r ,n, p\ge 0$ let
$$
F_p C^n(r) = \{ c \in F_p C^n \ | \ dc \in F_{p+r} C^{n+1} \}.
$$  
Thus, $F_p C^n = F_p C^n(0) \supset F_p C^n(1) \supset \ldots$.
We define subspaces of $F_p C^{p+q}$: first,
$$
Z_r^{pq} := F_p C^{p+q}(r) + F_{p+1} C^{p+q} \supset B_r^{pq} := d F_{p-r+1} C^{p+q-1}(r-1) + F_{p+1} C^{p+q}.
$$
Our assumption on the filtration means that for $r$ sufficiently large, $Z_r^{pq}
= Z_{\infty}^{pq}$ and $B_r^{pq} = B_{\infty}^{pq}$, where by definition
$$
Z_{\infty}^{pq} := F_p Z^{p+q} + F_{p+1} C^{p+q} \supset B_{\infty}^{pq} := F_p B^{p+q} + F_{p+1} C^{p+q}.
$$
There are inclusions of these spaces:
$$
0 = B_0^{pq} \subset  \ldots \subset B_r^{pq} \subset B_{r+1} \ldots \subset B_{\infty}^{pq} \subset Z_{\infty}^{pq} \subset \cdots
\subset Z_r ^{pq}\subset Z_{r-1}^{pq} \subset \ldots \subset Z_0 ^{pq}= E_0^{pq}.
$$

The $r$-th page of the spectral sequence is defined to be
\begin{equation} \label{e.pager1}
E_{r}^{pq} = Z_{r}^{pq} /B_{r}^{pq}  = 
\frac{F_p C^{p+q}(r) + F_{p+1}C^{p+q}}{dF_{p-r+1} C^{p+q-1}(r-1) + F_{p+1}C^{p+q}}.
\end{equation}
An equivalent definition of the $r$-th page is sometimes more convenient.
Given subspaces $A,B,C$ of a vector space $V$, with $A \supset B$, we have
a natural isomorphism
\begin{equation} \label{e.vsiso}
\frac{A}{ (A \cap C) + B } \to \frac{A+C}{B+C}.
\end{equation}
Applied to our situation, with $A = F_p C^{p+q}(r)$, $B=dF_{p-r+1} C^{p+q-1}(r-1)$,
$C = F_{p+1}C^{p+q}$, we find that $A \cap C = F_{p+1}C^{p+q}(r-1)$, so we obtain
a natural isomorphism
\begin{equation}  \label{e.pager2}
E_{r}^{pq} \cong
\frac{F_p C^{p+q}(r)}{G^{pq}_r}.
\end{equation}
where $G^{pq}_r = F_{p+1}C^{p+q}(r-1) + dF_{p-r+1} C^{p+q-1}(r-1)$.
Using this second description, we define a differential
$d_r: E_{r}^{pq} \to E_{r}^{p+r,q-r+1} $ by
$$
d_r (c + G^{pq}_r) = dc + G^{p+r,q-r+1}_r.
$$

The following proposition describes one of the key properties of spectral
sequences. 

\begin{Prop} \label{p.keyspectral}
Under the surjection $\pi_p: Z_r^{pq} \to E_r^{pq}$, the inverse image of
$\ker d_r$ (resp.~$\im d_r$) is $Z_{r+1}^{pq}$ (resp.~$B_{r+1}^{pq}$).  Hence
$\pi_p$ induces an isomorphism 
$$
E_{r+1}^{pq} = Z_{r+1}^{pq}/B_{r+1}^{pq} \to H(E_r,d_r)^{pq}.
$$
\end{Prop}

We omit the proof.

\begin{Rem}\label{r.altdefdone}
For $r=0$ and $r=1$, the definitions \eqref{e.pager1} and \eqref{e.pager2}
of the $r$-th page are identical.  In particular, $G^{pq}_r = B^{pq}r$,
and $d_r(z + B_r^{pq})=dz + B_r^{pq}$.  However, the analogous
assertion is false for $r>1$.
\end{Rem}

The spectral sequence is said to degenerate at $E_r$ if for all $p,q$, we have
$Z_r^{pq} = Z_{\infty}^{pq} $ and $B_r^{pq} = B_{\infty}^{pq} $.  This is equivalent
to the vanishing of all differentials $d_s$, $s \ge r$, and implies that
$E_r^{pq} = E_{\infty}^{pq} $.

Observe that
$$
\mbox{gr}^p H^{p+q}(C) \cong \frac{F_p Z^{p+q} + B^{p+q}}{F_{p+1} Z^{p+q} + B^{p+q}} \cong 
\frac{F_p Z^{p+q}}{F_p B^{p+q} + F_{p+1} Z^{p+q}}.
$$
The composition
$$
F_p Z^{p+q}  \to Z_{\infty}^{pq} \to E_{\infty}^{pq}
$$
induces an isomorphism $\mbox{gr}^p H^{p+q}(C) \to E_{\infty}^{pq}$.

Suppose that $C$ has a graded algebra structure such that
the differential $d$ satisfies $d(c c') = (dc) c' + (-1)^p c (dc')$ for $c \in C^p, c' \in C$.
Then $H(C)$ is a graded algebra, and $\mbox{gr }H(C)$ is a bigraded algebra.
Moreover, each $E_r$ has the structure of a bigraded algebra, and the map
$\mbox{gr }H(C) \to E_{\infty}$ is an isomorphism of bigraded algebras.

\subsection{Edge maps} \label{ss.edge}
Suppose that $F_{p+1} C^p = 0$ for all $p$.  Suppose also that $r \ge 2$; in this
case $F_p C^p(r) = F_p Z^p$.  Consider the map
\begin{equation} \label{e.edge1}
\gre_p:E_r^{p0} = \frac{F_p Z^p}{d F_{p-r+1}C^{p-1}(r-1)} \to E_{\infty}^{p0} = \frac{F_p Z^p}{F_p B^p} = \mbox{gr}^p H^p(C)
\subset H^p(C).
\end{equation}
The first map $E_r^{p0} \to E_{\infty}^{p0}$ is surjective.  If the spectral sequence degenerates at $E_r$, we obtain
$$
E_r^{p0} = \mbox{gr}^p H^p(C) \hookrightarrow H^p(C).
$$

Next, making use of the facts that if $r >0$ then $F_{1-r}C^q = C^q$, 
we have an edge map
\begin{equation} \label{e.edge2}
\gre_q: H^q(C) = \frac{Z^q}{B^q} \to \frac{Z^q}{B^q + F_1 Z^q} = E_{\infty}^{0q} \hookrightarrow
E_r^{0q} = \frac{F_0 C^q(r)}{d F_0C^{q-1}(r-1)  + F_1 C^q(r-1)}.
\end{equation}
Here the first map is surjective.  If the spectral sequence degenerates at $E_r$, we obtain
a surjection
$$
H^q(C) \to \mbox{gr}^0 H^q(C) = E_r^{0q} .
$$

\subsection{The product structure}
Let $A^p = E_{\infty}^{p0}$, $A= \oplus_p A^p$, $B^q =  E_{\infty}^{0q}$, $B= \oplus B^q$.
Write $A^{+} = \oplus_{p>0} A^p$.
We endow the tensor product $A \otimes B$ with an algebra structure such that
$$
(a_1 \otimes b_1) \cdot (a_2 \otimes b_2) = (-1)^{ q_1 p_2} a_1 a_2 \otimes b_1 b_2,
$$
for $b_1 \in B^{q_1}$, $a_2 \in A^{p_2}$.

The edge maps give an inclusion $A \hookrightarrow H(C)$,
$A^p = \mbox{gr}^p H^p(C) \subset H^p(C)$.  Similarly, the edge maps give
a surjection $H(C) \to B$, $H^q(C) \to \mbox{gr}^0 H^q(C) = B^q$.  
Let $J$ denote the kernel of the map $H(C) \to B$; then $J = \oplus_q F_1 H^q(C)$.

\begin{Prop} \label{p.spectralideal}
Suppose that the multiplication map $A \otimes B \to E_{\infty}$ is an algebra
isomorphism.  Then:

(1) The ideal $J$ is equal to the ideal of $H(C)$ generated by $A^{+}$.

(2) If $ b_1, \ldots b_n $ are homogeneous elements of $H(C)$ whose images
under the map $H(C) \to B$ form a basis of $B$,
then these elements form an $A$-module basis
of $H(C)$.  
Hence if $\dim B = n$ is finite, then $H(C)$ is a free $A$-module of rank $n$.
\end{Prop}

\begin{proof}
(1) Let $J'$ denote the ideal $A^+ H(C)$.  We want to show that $J' = J$.  Since
$A^+ \subset F_1 H(C)$, we have $J' \subset J$.   For the reverse inclusion,
it suffices to show that if $h \in F_p H^q(C)$ with $p \ge 1$, then
 $h \in J'$.  We use downward induction on $p$.  If $p >q$ then
$F_p H^q(C) = 0$ so the result holds.  Now suppose the result holds for elements
of $F_{p+1}H^q(C)$.  Let $\bar{h}$ denote the image of $h$ in $E_{\infty}^{p,q-p}$.
By hypothesis we can write $\bar{h} = \sum a_q \gb_q$ for some elements 
$a_q \in A^p$, $\gb_q \in B^{q-p}$.  Choose elements $b_q \in H^{q-p}(C)$
mapping to $\gb_q \in B$.  Then $h - \sum a_q b_q\in F_{p+1}H^q(C)$  is
in $J'$ by the inductive hypothesis.  Since $\sum a_q b_q \in J'$, we conclude
$h \in J'$, as desired.

(2) Let $b_1, \ldots, b_n$ be as in the statement of the proposition; let $d_q$ denote
the degree of $b_q$, and let $\bar{b}_q $ denote its image in $B^{d_q} \cong F_0 H^{d_q}(C)/F_1 H^{d_q}(C)$.
The argument that the $b_q$ generate $H(C)$ as an $A$-module is similar to the proof of
(1), and we omit the details.  To show that the $b_q$ are linearly independent over $A$,
suppose we have a relation 
$\sum_q a_q b_q = 0$ with not all $a_q = 0$.  We may assume that each term $a_q b_q$
has the same degree, which we denote $n$; then the degree of $a_q$ is $c_q := n-d_q$.
Let $c$ denote the smallest of the $c_q$ for which $a_q \neq 0$.  Then each term
$a_q b_q$ lies in $F_c H^n(C)$.  By definition of the product on $E_{\infty}$, the
image of $\sum_q a_q b_q$ in $E_{\infty}^{c, n-c} \cong F_c H^n(C) /F_{c+1} H^n(C)$
is $\sum_{c_q = c} a_q \bar{b}_q$.  Since the multiplication map $A \otimes B \to E_{\infty}$
is an isomorphism, this implies $\sum_{c_q = c} a_q \otimes \bar{b}_q$ is zero in $A \otimes B$.
As the $\bar{b}_q$ are linearly independent, this implies that all the $a_q$ with
$c_q = c$ must be $0$, contradicting our choice of $c$.  We conclude that the
$b_q$ are linearly independent, as desired.
\end{proof}

\section{The relative Hochschild-Serre spectral sequence} \label{s.hochserre}
The following theorem
shows the existence of the
Hochschild-Serre spectral sequence and identifies the edge maps.

\begin{Thm} \label{t.hochserre}
Let $\fg$ be a Lie algebra.  Let $\fk$ be a subalgebra of $\fg$, reductive in $\fg$,
and let $I$ be an ideal of $\fg$.
Let $I_{\fk} = I \cap \fk$.
Let $M$ be a $\fg$-module.  

(1) There is a spectral sequence converging to $H^{p+q}(\fg,\fk;M)$, and an
isomorphism $\psi$:
$$
E_2^{pq} \stackrel{\psi}{\rightarrow} H^p(\fg/I, \fk/I_{\fk}; H^q(I, I_{\fk};M)).
$$

(2) The edge morphism $E_2^{p0} \to H^p(\fg,\fk;M)$ corresponds under the isomorphism
$\psi$ to the composition
$$
H^p(\fg/I, \fk/I_{\fk}; M^I) \to H^p(\fg, \fk; M^I) \to H^p(\fg, \fk;M),
$$
where the first map is the pullback induced by the projection $(\fg, \fk) \to (\fg/I, \fk/I_{\fk})$,
and the second map is induced by the $\fg$-module map $M^I \to M$.

(3) The edge morphism $H^q(\fg, \fk;M) \to E^{0q}_2$ corresponds under the isomorphism
$\psi$ to the pullback
$$
i^*:H^q(\fg, \fk;M) \to H^q(I,I_{\fk};M)^{\fg/I}.
$$ 
\end{Thm}

Part (1) of this theorem is proved in Section \ref{s.existspectral}.
Part (2) is Proposition \ref{p.bottomedge}, and part (3) is
Proposition \ref{p.leftedge}.  In Section \ref{s.pairings} 
we will show that the spectral sequence is
compatible with pairings of representations.

\subsection{Some notation} \label{s.existspectral}
We follow the notational conventions of Definition \ref{d.streamline}.
We have
$$
F_p C^{p+q} \stackrel{\pi_p}{\rightarrow} E_0^{pq}  \stackrel{\spover}{\rightarrow} C^p(C^q).
$$
The composition is $s_p$, and by Lemma \ref{l.sj} it induces the isomorphism $\spover$.
Let $\stotal$ denote $\oplus_p \spover$.

\subsection{The differential $d_0$}
We begin by calculating the differential $d_0$. 

\begin{Prop} \label{p.d0}
Under the isomorphism $\stotal$, the differential
$d_0: E_0^{pq} \to E_0^{p,q+1}$ corresponds to $(-1)^{p} d_v$.
\end{Prop}

\begin{proof}
Let $e \in E_0^{pq}$, and choose $\hat{e} \in F_p C^{p+q}$
satisfying $\pi_p\hat{e}= e$.   By definition,
$d_0 e = \pi_p de$, so $\spover d_0 e = s_p d \hat{e}$.
As $s_p$ is the restriction of $S_p$, we apply Lemma
\ref{l.differential} and find
$$
\spover d_0 e = s_p d \hat{e} = S_p d \hat{e} = 
d_h (S_{p-1} \hat{e}) + (-1)^{p} d_v (S_{p} \hat{e}).
$$
As $\hat{e} \in F_p C^{p+q}$, we have $S_{p-1} \hat{e} = 0$ by Lemma \ref{l.sj}.
Therefore,
$$
\spover d_0 e = (-1)^{p} d_v (S_{j} \hat{e}) = (-1)^p d_v s_p \hat{e} = (-1)^p d_v  \spover e,
$$
proving the proposition.
\end{proof}

The following diagram summarizes some of the relationships between the groups
we are considering.  The vertical arrows are inclusions, and the horizontal maps
$\ga$ and $\gb$ (defined by this diagram)
are the surjections from cocycles to cohomology.

$\begin{CD}
F_p C^{p+q} @>{s_p}>> C^p(C^q) \\
                        @.                  @AAA \\
                        @.                 C^p(Z^q) @>{\ga}>> C^p(H^q) \\
                        @.                  @.                                  @AAA  \\
                        @.                  @.                                   Z^p(H^q) @>{\gb}>> H^p(H^q).
\end{CD}$                       
                       
As a consequence of Proposition \ref{p.d0}, we have:

\begin{Cor} \label{c.e1}
Under the surjective map $s_p$, 
the inverse image of $C^p(Z^q)$ (resp.~$C^p(B^q)$) is $Z_1^{pq}$
(resp.~$B_1^{pq})$.  Hence the map $E_1^{pq} \to C^p(H^q)$
defined by $c + B_1^{pq} \mapsto \ga s_p (c)$ is an isomorphism.
\end{Cor}

\begin{proof}
We have
$$
s_p^{-1} (C^p(Z^q)) = \pi_p^{-1}(\ker d_0) = Z_1^{pq},
$$
where the first equality is by Proposition \ref{p.d0}, and the second
by Proposition \ref{p.keyspectral}.  A similar argument works to show
$s_p^{-1} (C^p(B^q)) = B_1^{pq}$.  
\end{proof}

For $e = c+B_1^{pq} \in E_1^{pq}$, write $\phi(e) = \ga s_p(c) \in C^p(H^q)$.

\begin{Prop} \label{p.d1}
Under the isomorphism $\phi: E_1^{pq} \to C^p(H^q)$, the spectral
sequence differential
$d_1$ corresponds to the Lie algebra cohomology differential $d_+$,
i.e., $\phi( d_1 e) = d_+ \phi (e)$.
\end{Prop}

\begin{proof}
Let $e \in E_1^{pq}$, and let $\phi(e)= \ga(z)$, with $z\in C^p(Z^q)$.
Let $\tilz \in F_p C^{p+q}$ be the lift of $z$ defined in Definition \ref{d.tildemap}.   Then
$s_p(\tilz)=z$, so by Corollary \ref{c.e1},
$\tilz \in Z_1^{pq}$.  Observe that $e = \tilz + B_1^{pq}$: this
follows by Corollary \ref{c.e1}, since 
$$
\phi(\tilz + B_1^{pq}) = \ga(s_p \tilz) = \ga(z) = \phi(e).
$$
By Remark \ref{r.altdefdone}, $d_1e = d\tilz + B_1^{p+1,q}$, and therefore
$$
\phi(d_1 e)=\phi(d\tilz + B_1^{p+1,q})=\alpha s_{p+1}(d\tilz) = \ga(d_+z) = d_+ \ga(z) = d_+\phi(e),
$$
where the third equality is by Corollary \ref{c.sp+1}, and the fourth is by definition of
the map $d_+$ on cohomology.  This proves the result.
\end{proof}



\begin{Thm} \label{t.e2}
If $c \in Z_2^{pq}$, then $\ga s_p (c) \in Z^p(H^q)$.  The map
$Z_2^{pq} \to H^p(H^q)$ given by $c \mapsto \beta \ga s_p(c)$ is surjective with kernel
$B_2^{pq}$.  Hence the induced map
$$
\psi: E_2^{pq} \to H^p(H^q)
$$
defined by
$$
\psi(c + B_2^{pq}) = \gb \ga s_p(c),
$$
is an isomorphism.
\end{Thm}

\begin{proof}
We have 
$$
Z_1^{pq} \stackrel{\pi_p}{\rightarrow} E_1^{pq} \stackrel{\phi}{\rightarrow} C^p(H^q),
$$
where $\pi_p$ is surjective and $\phi$ is an isomorphism.
By Proposition \ref{p.d1}, $\phi:\ker(d_1) \to Z^p(H^q)$ and $\phi:\im(d_1)
\to B^p(H^q)$ are isomorphisms.
Thus, Proposition \ref{p.keyspectral} implies that under $\phi \circ \pi_p$,
the inverse image of $Z^p(H^q)$ is $Z_2^{pq}$ and the
inverse image of $B^p(H^q)$ is $B_2^{pq}$.
The result follows.
\end{proof}

\subsection{Edge maps}
In this section we show that the edge maps are compatible with
maps defined using the functorial properties of Lie algebra cohomology.

Since $H^0(I,I_{\fk};M) = M^I$, we have
$$
H^p(H^0) = H^p(\fg/I, \fk/I_{\fk}; H^0(I,I_{\fk};M)) = H^p(\fg/I, \fk/I_{\fk}; M^I).
$$
There is a natural morphism
$$
\eta: H^p(H^0) \to H^p(\fg, \fk;M)
$$
defined as the composition
$$
H^p(\fg/I, \fk/I_{\fk}; M^I) \to H^p(\fg, \fk; M^I) \to H^p(\fg, \fk;M),
$$
where the first map is the pullback induced by the projection $(\fg, \fk) \to (\fg/I, \fk/I_{\fk})$,
and the second map is induced by the $\fg$-module map $M^I \to M$.
More concretely, we can view $C(\fg/I,\fk/I_{\fk};M^I)$ as the subspace
of elements $f \in C(\fg, \fk;M)$ such that $f$ vanishes when any argument is
in $I$, and such that the image of $f$ lies in $M^I$.  We claim that
$$
F_p Z^p(\fg, \fk;M) =  Z^p(\fg/I, \fk/I_{\fk}; M^I).
$$
Since by definition of the filtration any element $g$ of the left hand side vanishes
when any argument is in $I$, to verify the claim we only need to check that any
such element has image in $M^I$.  This follows because if $x_0 \in I$, then
$$
0 = dg(x_0, \ldots, x_p) = x_0 g(x_1, \ldots, x_p).
$$
Let $f \in Z^p(\fg/I, \fk/I_{\fk}; M^I)$.  The map $\eta$ takes the class of $f$ in
$H^p(\fg/I, \fk/I_{\fk}; M^I)$ to the class of the same element $f$, but now in
$H^p(\fg, \fk;M)$.  We see that the image of $\eta$ lies in $F_p H^p(\fg, \fk;M)$.

There is another morphism $H^p(H^0) \to H^p(\fg, \fk;M)$ defined
as the composition of $\psi^{-1}$ with the edge morphism $\gre_p$:
$$
H^p(H^0) \stackrel{\psi^{-1}}{\rightarrow} E^{p0}_2 \stackrel{\gre_p}{\rightarrow}
H^p(\fg, \fk;M).
$$
The next proposition shows that this morphism coincides with $\eta$.

\begin{Prop} \label{p.bottomedge}
The following diagram commutes:
$$
\xymatrix{
E_2^{p0} \ar[r]^{\gre_p} \ar[d]^{\psi}
&H^p(\fg, \fk;M) \\
H^p(H^0) \ar@{>}^{\eta}[ur] & }
$$
\end{Prop}
\begin{proof}
As observed in equation \eqref{e.edge1},
$$
E_2^{p0} =\frac{ F_p Z^p(\fg, \fk;M)}{d F_{p-1} C^{p-1}(1)}.
$$
Let $f \in F_p Z^p(\fg, \fk;M) = Z^p(\fg/I, \fk/I_{\fk}; M^I)$.   Then $\gre_p$
takes the class of $f$ in $E_2^{p0}$ to the class of $f$ in $H^p(\fg, \fk;M)$.
On the other hand,
$s_p f$
is just $f$, now viewed as an element of $C^p(\fg/I,\fk/I_{\fk}; C^0(I,I_{\fk};M))$,
and $\alpha s_p f$ is again $f$, but now viewed as an element of
$Z^p(\fg/I,\fk/I_{\fk}; M^I)$.  Thus, $\psi$ takes the class of $f$ in $E_2^{p0}$
to the class of $f$ in $H^p(\fg/I,\fk/I_{\fk}; M^I)$.  The discussion preceding the proposition shows that applying $\eta$ to this
yields the class of $f$ in $H^p(\fg, \fk;M)$.  We conclude that the diagram commutes,
as claimed.
\end{proof}

We now consider the other edge morphism.  By Proposition \ref{p.pullbackinvariant},
the inclusion $i: (I,I_{\fk}) \to (\fg, \fk)$
induces a pullback 
$$
i^*:H^q(\fg, \fk;M) \to H^q(I,I_{\fk};M)^{\fg/I} = H^0(\fg/I;H^q(I,I_{\fk};M)) = H^0(H^q).
$$ 
On the other hand, we have a morphism defined as the composition
$$
H^q(\fg, \fk;M) \stackrel{ \gre_q}{\rightarrow} E^{0q}_2 \stackrel{\psi}{\rightarrow} H^0(H^q).
$$
The next proposition shows that these two morphisms agree.

\begin{Prop} \label{p.leftedge}
The following diagram commutes:
$$
\xymatrix{
H^q(\fg, \fk;M)\ar[r]^{\gre_q} \ar[dr]^{i^*}
&E_2^{0q} \ar[d]^{\psi} \\
& H^0(H^q) .}
$$
\end{Prop}
\begin{proof}
Let $c \in Z^q(\fg, \fk;M)$.  By Definition (see \eqref{e.edge2}), 
the edge homomorphism $\gre_q$ takes
the class of $c$ in $H^q(\fg, \fk;M)$ to the class of $c$ in $E_2^{0q}$,
and $\psi$ takes this class to $\gb \ga s_0(c)$.  By definition \ref{d.rj},
$s_0(c) \in C^0(C^q) = C^q(I,I_{\fk};M)^{\fk/I_{\fk}}$ is the pullback $i^*c$.  In addition,
$i^*c \in Z^q(I,I_{\fk};M)^{\fk/I_{\fk}}$, and
$\ga s_0(c)$ is the cohomology class of $s_0(c) = i^*c$ in $C^0(H^q) = H^q(I,I_{\fk};M)^{\fk/I_{\fk}}$.  
Further, $\ga s_0(c)$ lies in $Z^0(H^q) = H^0(H^q) = H^q(I,I_{\fk};M)^{\fg/I}$, and $\beta: Z^0(H^q) \to H^0(H^q)$
is the identity map.  We conclude that the composition $\psi \circ \gre_q$ takes the cohomology
class of $c$ to the cohomology class of $i^*c$, so the diagram commutes.
\end{proof}

\subsection{Pairings}  \label{s.pairings}
A pairing of $\fg$-modules induces a pairing on spectral sequences.  In this section
we show that the basic result about these pairings (\cite{HS:53}, Theorem 5) extends
to the relative situation.

Given a $\fg$-module $M$, write $E_r^{pq}(M)$ for the corresponding
Hochschild-Serre spectral sequence.  Write 
\begin{equation} \label{e.cohmodule}
H^p(H^q(M)) = H^p(\fg/I, \fk/I_{\fk}; H^q(I,I_{\fk};M)),
\end{equation}
 and
let $\psi_M: E_2^{pq}(M) \to H^p(H^q(M))$ denote the isomorphism of
Theorem \ref{t.e2}.

Suppose $M,N$ and $P$ are $\fg$-modules with a $\fg$-module map
$M \otimes N \to P$.  There is a ``cup product" pairing
$$
C^p(\fg;M) \otimes C^q(\fg;N) \to C^{p+q}(\fg;P)
$$
taking $a \otimes b$ to $a \cup b$ (the formula is given in \cite{HS:53}, p. 592).
As verified by Hochschild and Serre, this product is compatible with the differential $d$ and the action $\gt_z$
for $z \in \fg$, in that
\begin{equation} \label{e.diffcup}
d(a \cup b) = da \cup b + (-1)^p a \cup db
\end{equation}
and 
$$\gt_z(a \cup b) = \gt_z a \cup b + a \cup \gt_z b.
$$ 
The cup product pairing
induces a pairing on the spaces of relative cochains:
$$
C^p(\fg, \fk;M) \otimes C^q(\fg, \fk;N) \to C^{p+q}(\fg, \fk;P).
$$
Indeed, let $a \in C^p(\fg, \fk;M)$
and $b \in C^q(\fg, \fk;N)$.  If $z \in \fk $, then 
$\gt_z(a \cup b) = \gt_z a \cup b + a \cup \gt_z b = 0$, and
$i_z(a \cup b) = 0$ (this is immediate from the cup product
formula).  The compatibility with the differential implies that
the cup product pairing descends to cohomology, yielding a pairing
$$
H^p(\fg, \fk;M) \otimes H^q(\fg, \fk;N) \to H^{p+q}(\fg, \fk;P).
$$

\begin{Lem} \label{l.cup}
There is a pairing
$$
E_r^{p_1,q_1}(M) \otimes E_r^{p_2,q_2}(N) \to E_r^{p_1+p_2,q_1+q_2}(N).
$$
taking $[a] \otimes [b]$ to $[a \cup b]$ (the brackets denote the class of a cochain).
\end{Lem}

\begin{proof}
Write $A^n = C^n(\fg, \fk;M)$, $B^n = C^n(\fg, \fk;N)$, $D^n = C^n(\fg, \fk;P)$.  First, observe that
the cup product satisfies
\begin{equation} \label{e.cup1}
F_{p_1} A^{p_1+ q_1} \otimes F_{p_2} B^{p_2+q_2} \to F_{p_1+p_2} D^{p_1+p_2+q_1+q_2}.
\end{equation}
Next, the property $d(a \cup b) = da \cup b + (-1)^p a \cup db$ and
\eqref{e.cup1} imply that
\begin{equation} \label{e.cup2}
F_{p_1} A^{p_1+ q_1}(r) \otimes F_{p_2} B^{p_2+q_2}(r) \to F_{p_1+p_2} D^{p_1+p_2+q_1+q_2}(r).
\end{equation}
The lemma follows in a straightforward way
from equations \eqref{e.cup1} and \eqref{e.cup2}.
\end{proof}

The pairing is compatible with the differential $d_r$:

\begin{Lem}
Let $e \in E_r^{p_1+q_1}(M)$ and $e' \in  E_r^{p_2+q_2}(N)$.  Then
$$
d_r(e \cup e') = (d_r e) \cup e' + (-1)^{p_1+q_1} e \cup (d_r e').
$$
\end{Lem}

\begin{proof}
The differential $d_r$ is calculated by choosing representative
cochains and applying the differential $d$.  The lemma is then
a consequence of \eqref{e.diffcup}.
\end{proof}

There is a pairing
$$
H^{q_1}(I,I_{\fk};M) \otimes H^{q_2}(I,I_{\fk};N) \to H^{q_1+q_2}(I,I_{\fk};P).
$$
Since this is a $\fg/I$-module pairing, we obtain a corresponding
pairing in $(\fg/I,\fk/I_{\fk})$-cohomology (with notation as in \eqref{e.cohmodule}):
\begin{equation} \label{e.pairing}
H^{p_1}(H^{q_1}(M) ) \otimes H^{p_2}(H^{q_2}(N) )   \stackrel{\cup}{\rightarrow} 
 H^{p_1+p_2}(H^{q_1+q_2}(P)).
\end{equation}
Note that this pairing is derived from a pairing on the level of cochains.

The next result relates this pairing with the spectral sequence pairing,
extending Theorem 5 of \cite{HS:53} to the relative setting.

\begin{Prop} \label{p.comparepairing}
Let $e \in E_r^{p_1+q_1}(M)$ and $e'  \in  E_r^{p_2+q_2}(N)$.
Then
$$
\psi_P(e \cup e') = (-1)^{p_2 q_1} \psi_M(e) \cup \psi_N(e').
$$
\end{Prop}

\begin{proof}
Keep the notation of the proof of Lemma \ref{l.cup}.  Let
$a \in F_{p_1} A^{p_1+q_1}$ and $b \in F_{p_2}B^{p_2+q_2}$
represent $e$ and $e'$, respectively.  It suffices to prove
that the pairing of cochains satisfies
$$
s_{p_1}(a) \cup s_{p_2}(b) = (-1)^{p_2 q_1} s_{p_1 + p_2}(a \cup b).
$$
This follows from the calculation in the proof of Theorem 3 of \cite{HS:53}.
\end{proof}

\subsection{A tensor product decomposition}
The natural pairing of the trivial representation $\F$ with itself induces a product
on $E_2(\F)$, making this space a bigraded ring.
In this subsection, we describe a tensor product decomposition of 
$E_2(\F)$, under
a certain cohomological hypothesis.
For simplicity, if the coefficient representation is the trivial representation
$\F$, we will generally omit it from the notation.  Thus, for example, we write
$C^n(\fg, \fk)$ and $H^n(\fg, \fk)$ for 
  $C^n(\fg, \fk;\F)$ and $H^n(\fg, \fk;\F)$, respectively, and $E_2$ for $E_2(\F)$.

 Unless otherwise stated, we will assume
 $H(\fg/I,\fk/I_{\fk}) \otimes H(I,I_{\fk})$ is given the product structure defined by
 \begin{equation} \label{e.signedproduct}
 (a_1 \otimes b_1)(a_2 \otimes b_2) = (-1)^{p_2 q_1} (a_1 \cup a_2) \otimes (b_1 \cup b_2),
 \end{equation}
where $a_2 \in H^{p_2}(\fg/I,\fk/I_{\fk})$ and $b_1 \in H^{q_1}(I,I_{\fk})$.  
We will denote the product defined by the same formula except without the factor of
$(-1)^{p_2 q_1}$ by $*$, and refer to this as the ``unsigned product".

\begin{Prop} \label{p.productbk}
Suppose that for any $\fg/I$-module $V$,
the inclusion $V^{\fg/I} \hookrightarrow V$ induces an isomorphism
\begin{equation} \label{e.includeiso}
H(\fg/I,\fk/I_{\fk};V^{\fg/I}) \to H(\fg/I,\fk/I_{\fk};V).
\end{equation}
Then there exists an algebra isomorphism
$$
\Psi: E_2 \cong H(\fg/I,\fk/I_{\fk}) \otimes H(I,I_{\fk})^{\fg/I}
$$
such that if $e_1 \in E_2^{p0}$ and $e_2 \in E_2^{0q}$, then
\begin{equation} \label{e.compat1}
\Psi(e_1) = \psi(e_1) \otimes 1 
\end{equation}
\begin{equation} \label{e.compat2}
\Psi(e_2) = 1 \otimes \psi(e_2).
\end{equation}
\end{Prop}

\begin{proof}
First, there is a natural isomorphism
$$
f: C^p(\fg/I,\fk/I_{\fk}) \otimes H^q(I,I_{\fk})^{\fg/I} \to C^p(\fg/I,\fk/I_{\fk}; H^q(I,I_{\fk})^{\fg/I})
$$
defined by
$$
f(c \otimes \gd)(x_1, \ldots, x_p) = c(x_1, \ldots, x_p) \gd.
$$
This map induces a cohomology
isomorphism (also denoted $f$)
$$
f: H^p(\fg/I,\fk/I_{\fk}) \otimes H^q(I,I_{\fk})^{\fg/I} \to H^p(\fg/I,\fk/I_{\fk}; H^q(I,I_{\fk})^{\fg/I}).
$$
The space $H(\fg/I,\fk/I_{\fk}; H(I,I_{\fk})^{\fg/I})$ has a ring structure.  In fact, the
definition of the product 
(which is made using the cohomology pairing induced by a pairing of representations)
is such that if we equip $H(\fg/I, \fk/I_{\fk}) \otimes H(I,I_{\fk})^{\fg/I} $ with the 
unsigned product, then $f$ is an algebra isomorphism.
The isomorphism of \eqref{e.includeiso} also respects the algebra structure,
so the composition
$$
F: H(\fg/I,\fk/I_{\fk}) \otimes H(I,I_{\fk})^{\fg/I} \stackrel{f}{\rightarrow}
 H(\fg/I,\fk/I_{\fk}; H(I,I_{\fk})^{\fg/I}) \to H(\fg/I,\fk/I_{\fk}; H(I,I_{\fk}))
$$
is an algebra isomorphism.  

We now consider
$$
F: H^p(\fg/I,\fk/I_{\fk}) \otimes H^q(I,I_{\fk})^{\fg/I} \to H^p(\fg/I,\fk/I_{\fk}; H^q(I,I_{\fk}))
$$
in case $p=0$ or $q=0$.  If $p = 0$, then $H^0(\fg/I,\fk/I_{\fk}) = \F$
and $H^0(\fg/I,\fk/I_{\fk}; H^q(I,I_{\fk})) = H^q(I,I_{\fk})^{\fg/I} $, and we have
\begin{equation} \label{e.F1}
F(1 \otimes b) = b.
\end{equation}
On the other hand, if $q=0$, then $H^0(I,I_{\fk})^{\fg/I} = \F$ and
$H^p(\fg/I,\fk/I_{\fk}; H^0(I,I_{\fk})) = H^p(\fg/I,\fk/I_{\fk})$, and we have
\begin{equation} \label{e.F2}
F(a \otimes 1) = a.
\end{equation}

Let $\nu$ denote the inverse of $F$, and define $\Psi$ as the composition
$$
\Psi: E_2 \stackrel{\psi}{\rightarrow} H(\fg/I,\fk/I_{\fk}; H(I,I_{\fk})) \stackrel{\nu}{\rightarrow}
H(\fg/I,\fk/I_{\fk}) \otimes H(I,I_{\fk})^{\fg/I}.
$$
The map $\Psi$ is a vector space isomorphism (since it is the composition of two such).
Let $e_1 \in E_2^{p_1,q_1}$ and $e_2 \in E_2^{p_2,q_2}$.  Then, applying
Proposition \ref{p.comparepairing}, we see
\begin{eqnarray*}
\Psi(e_1 e_2) & = & \nu \psi (e_1 e_2)  =  (-1)^{p_2 q_1} \nu(\psi(e_1) \psi(e_2)) \\
  & = & (-1)^{p_2 q_1} \nu(\psi(e_1)) * \nu (\psi(e_2)) = \Psi(e_1) \Psi(e_2).
\end{eqnarray*}
 Hence $\Psi$ is compatible with products, so it is an algebra isomorphism.

Finally, since $\Psi = \nu \psi$, equation \eqref{e.compat1} states that
$\nu \psi(e_1) = \psi(e_1) \otimes 1 $.  Since $F$ and $\nu$ are inverses,
this is equivalent to the equation
$\psi(e_1) = F(\psi(e_1) \otimes 1)$, which follows from \eqref{e.F2}.
The verification of \eqref{e.compat2} is similar.
\end{proof}

\begin{Cor} \label{c.productbk}
Keep the assumptions of Proposition \ref{p.productbk}.

(1) There is an isomorphism
\begin{equation} \label{e.productbk}
H^*(\fg/I,\fk/I_{\fk}) \otimes H^*(I,I_{\fk})^{\fg/I} \to E_2
\end{equation}
taking $a \otimes b$ to $\psi^{-1}(a) \psi^{-1}(b)$.

(2) Suppose that the spectral sequence degenerates at $E_2$.  Then
the pullback map $\pi^*: H(\fg/I,\fk/I_{\fk}) \to H(\fg, \fk)$ is injective
and the pullback map $i^*: H(\fg, \fk) \to H^*(I,I_{\fk})^{\fg/I}$ is surjective.
The kernel of $i^*$ is the ideal $J$ generated by elements of the form $\pi^*(a)$,
where $a$ is a positive degree homogeneous element of $H(\fg/I,\fk/I_{\fk})$.
\end{Cor}

\begin{proof}
Part (1) holds because the map \eqref{e.productbk} is the inverse of the
map $\Psi$ from Proposition \ref{p.productbk}.  For part (2), the injectivity
of $\pi^*$ and surjectivity of $i^*$ follow because these maps are identified
with edge maps (cf. Propositions \ref{p.bottomedge} and \ref{p.leftedge})
and the corresponding edge maps have this property.  The statement about
the kernel of $i^*$ then follows from Proposition \ref{p.spectralideal}.
\end{proof}

\section{The Belkale-Kumar product} \label{s.bk}
In this section we apply
the Hochschild-Serre spectral sequence to prove Theorem \ref{t.borelextension},
which shows that the cohomology 
of a generalized flag variety,
equipped with the Belkale-Kumar product, 
has a structure analogous to the cohomology of a fiber
bundle.

Throughout this section
$\fg$ is a complex semisimple Lie algebra.
Let $G$ denote the corresponding adjoint group,
with $B \supset H$ a Borel and maximal torus, respectively. 
Let $P$ be a parabolic subgroup of $G$ which contains $B$, and
let $m = \dim H^2(G/P)$ (cohomology is taken with complex
coefficients).  For each $t = (t_1, \ldots, t_m) \in \C^m$, 
Belkale and Kumar construct a product
$\odot_t$ on the space $H^*(G/P)$, and show that it is closely
related to the geometric Horn problem.
In \cite{EG:11}, it was shown that the ring $(H^*(G/P),\odot_t)$ is isomorphic
to a certain relative Lie algebra cohomology ring.

To describe this ring we need some more notation.  We denote the
Lie algebra of an algebraic group by the corresponding fraktur letter.
Let $\fl$ be the Levi subalgebra of $\fp$ containing $\fh$.
Let $W$ and $W_P$ denote the Weyl groups of $\fg$ and $\fl$, respectively.
Let $ \{ \ga_1, \ldots, \ga_n \} \subset \fh^*$ be the simple
roots corresponding to the positive system for which the roots of $\fb$
are positive.  Given $M \subset \{ 1, \dots, n \}$, let $\fl_M$ be the
Levi subalgebra generated by $\fh$ and the root spaces
$\fg_{\pm \ga_i }$ for $i\in M$.  Let $\fu_{M,+}$ 
(resp.~$\fu_{M,-}$) be the subalgebra
spanned by the positive (resp.~negative) root spaces not contained in
$\fl_M$.  We assume that the simple roots are numbered so that $\fl= \fl_I$
for $I =\{ \ga_{m+1}, \ldots, \ga_n \}$.

Let $t = (t_1, \ldots, t_m) \in \C^m$ and let
$J(t)=\{ 1 \le q \le m : t_q \not= 0 \}$.  Let $K = I \cup J(t)$.

Given a subspace $V$ of $\fg$, let $V_{\gD}$ denote the image
of $V$ under the diagonal map $\fg \to \fg \times \fg$.
Let 
$$
\tilde{\fu}_K = ( \fu_{K, -} \times \{ 0 \}) + ( \{ 0 \} \times  \fu_{K, +}).
$$
Define a subalgebra $\fg_K$ of $\fg \times \fg$ by
$$
\fg_K = \fl_{K, \gD} + \tilde{\fu}_K.
$$
Then $\tilde{\fu}_K$ is an ideal of $\fg_K$, and
$\fg_K/ \tilde{\fu}_K \cong \fl_K$.  Moreover,
$\fg_K$  contains $\fl_{\gD}$
(since $\fl_K$ does). 

For $t=0$, we denote $\tilde{\fu}_K$ by $\tilde{\fu}$, and note
that $\tilde{\fu} = \fu_- \times \{ 0 \} + \{ 0 \} \times \fu$,
where $\fu$ is the nilradical of $\fp$, and $\fu_-$ is the
opposite nilradical.  Thus, for $t=0$, $\fg_K = \fl_{\gD} + \tilde{\fu}$,
and $H^*(\fg_K, \fl_{\gD}; \C) \cong H^*(\tilde{\fu})^L$. 

The following theorem describes the Belkale-Kumar
product in terms of relative Lie algebra cohomology.  The
generic case and the case $t=0$ are established in \cite{BeKu:06};
the general case is in \cite{EG:11}.

\begin{Thm}\label{t.egone}
\par\noindent (1) The rings $(H^*(G/P),\odot_t)$ and $H^*(\fg_K, \fl_{\gD};\C)$ are isomorphic.
\par\noindent (2) $(H^*(G/P),\odot_0) \cong H^*(\tilde{\fu})^L$.
\end{Thm}

We now apply the Hochschild-Serre spectral sequence for relative
Lie algebra cohomology.   By Theorem \ref{t.hochserre}, there
is a spectral sequence converging to $H^*(\fg_K, \fl_{\gD})$
with $E_2$-term
\begin{equation}\label{e.hochserre}
H^*(\fg_K/ \tilde{\fu}_K, \fl_{\gD};
H^*(\tilde{\fu}_K, \C)).
\end{equation}

Note that in this spectral sequence, the role of $I_{\fk}$ is played by
$\tilde{\fu}_K \cap \fl_{\gD} = 0$.

\begin{Prop} \label{p.hsdegen}
The above spectral sequence degenerates at the $E_2$-term.
\end{Prop}

\begin{proof}
By standard facts concerning spectral sequences, it suffices to
prove that the dimension of the $E_2$-term coincides with
$\dim(H^*(\fg_K, \fl_{\gD}))$, which by Corollary 3.18 of \cite{EG:11}
is $|W^P|$.
 For this, we may identify
$\fg_K/\tilde{\fu}_K = \fl_{K, \gD}$.
We decompose 
$$H^*(\tilde{\fu}_K, \C) = 
H^*(\tilde{\fu}_K, \C)^{L_K} \oplus \sum_i V_i,
$$
where the sum is over nontrivial irreducible representations of $L_K$
appearing in the space $H^*(\tilde{\fu}_K, \C)$.
For each $V_i$, there is a Hochschild-Serre spectral sequence
converging to $H^*(\fl_K, \fl ; V_i)$ with $E_2$-term,
$H^*(\fl_K/\fz_K, \fl/\fz_K; H^*(\fz_K,V_i))$. 
Since $\fz_K$ acts semisimply on $V_i$,
 $H^*(\fz_K, V_i)=0$ if $\fz_K$ acts nontrivially on $V_i$.
Hence the modules $H^*(\fz_K, V_i)$ are a direct sum of 
nontrivial $\fl_K/\fz_K$-modules. Since $\fl_K/\fz_K$ is semisimple,
Theorem 28.1 from \cite{CE:48} implies that 
$H^*(\fl_K/\fz_K, \fl/\fz_K; H^*(\fz_K, V_i))=0$.  Hence the $E_2$-term
\eqref{e.hochserre} coincides with
\begin{equation} \label{e.invhochserre}
H^*(\fl_K, \fl_{\gD}; 
H^*(\tilde{\fu}_K, \C)^{L_K}).
\end{equation}
This last space coincides with
$$
H^*(\fl_K, \fl ) \otimes H^*(\tilde{\fu}_K, \C)^{L_K}.
$$
It is well-known that $\dim H^*(\fl_K, \fl ) = \dim H^*(P_K/P) = |W_{P_K}/W_P|$, where $W_{P_K}$ is the Weyl group
of $L_K$ (see \cite{EG:11}, Equation (3.10)).  Also, by Theorem 5.14 from \cite{Kos:61}, $\dim H^*(\tilde{\fu}_K, \C)^{L_K}=  |W/W_{P_K}|$.
We conclude that the dimension of the $E_2$ term \eqref{e.hochserre} is
$$
|W_{P_K}/W_P| \cdot |W/W_{P_K}| = |W^P|,
$$ and the
proposition follows. 
\end{proof}

\begin{Thm} \label{t.e2ring}

\noindent (1) $H^*(\fl_K, \fl)$ is isomorphic to a subalgebra of $H^*(\fg_K, \fl_{\gD})$.
\par\noindent (2) Let $I_+ = \sum_{q > 0} H^q(\fl_K, \fl)H^*(\fg_K, \fl_{\gD})$,
where we identify $H^*(\fl_K, \fl)$ with its image in $H^*(\fg_K, \fl_{\gD})$.
 Then
$$H^*(\fg_K, \fl_{\gD})/I_+ \cong H^*(\tilde{\fu}_K, \C)^{L_K}.$$
\par\noindent (3) The cohomology ring $H^*(P_K/P)$ with the usual
cup product is isomorphic to a graded subalgebra $A$
of $(H^*(G/P), \odot_t)$.  Further, the ring
$(H^*(G/P_K), \odot_0) \cong
(H^*(G/P), \odot_t)/I_+$, where $I_+$ is the ideal of
$(H^*(G/P), \odot_t)$ generated by positive degree elements
of $A$.
\end{Thm}

\begin{proof} 
Let $p:\fg_K \to \fg_K/\tilde{\fu}_K \cong \fl_K$ be the
projection.  By Propositions \ref{p.bottomedge} and \ref{p.hsdegen} and 
Corollary \ref{c.productbk}, $p^*$
is an injective ring homomorphism, which proves (1).
Part (2) follows immediately from Corollary \ref{c.productbk}.
Part (3) follows from parts (1) and (2) and Theorem \ref{t.egone}. 
\end{proof}


\end{document}